\documentclass[11pt,reqno]{amsart}
\usepackage{srcltx}
\usepackage{empheq}
\usepackage{bm}
\usepackage{amsmath,amssymb,cases,color,mathtools}
\usepackage{hyperref}
\usepackage[capitalise]{cleveref}
\usepackage{vmargin}
\setmarginsrb{1.5cm}{1cm}{1.5cm}{3cm}{1cm}{1cm}{2cm}{2cm}
\usepackage{bigints}
\usepackage{relsize}
\usepackage{stmaryrd}

\usepackage{amsmath, amssymb,amscd, }
\usepackage{amsfonts}
\usepackage{mathrsfs}
\usepackage{graphicx}
\usepackage{xcolor}

 \usepackage{upref}
\hypersetup{linkcolor=blue, colorlinks=true,citecolor = red}
\newtheorem {theorem}{Theorem}[section]
\newtheorem {lemma}{Lemma}[section]
\newtheorem {definition}{Definition}[section]
 
\newtheorem {remark}{Remark}[section]

\theoremstyle{definition}

\theoremstyle{remark}
\numberwithin{equation}{section}
 
\newcommand{\mc}{\mathcal}
\newcommand{\R}{\mathbb{R}}
\newcommand{\N}{\mathbb{N}}
\newcommand{\dx}{\,{\rm d} {x}}
\newcommand{\dS}{\,{\rm d} {S}}

\begin{document}

\title[Self-propelled Navier-slip]{Steady Self-Propelled Motion of a Rigid Body in a Viscous Fluid with Navier-Slip Boundary Conditions}

 \author{\v S\'arka Ne\v casov\'a $^1$\and Arnab Roy$^{2,3}$ \and Ana Leonor Silvestre$^{4}$}
  \thanks{ {\v{S}}.N. has been supported by the  Praemium Academiae of {\v{S}}. Ne{\v{c}}asov{\'{a}}. The Institute of Mathematics, CAS is supported by RVO:67985840. A.R is supported by the Grant RYC2022-036183-I funded by MICIU/AEI/10.13039/501100011033 and by ESF+. A.R has been partially supported by the Basque Government through the BERC 2022-2025 program and by the Spanish State Research Agency through BCAM Severo Ochoa CEX2021-001142-S and through project PID2023-146764NB-I00 funded by MICIU/AEI/10.13039/501100011033 and cofunded by the European Union. A. L. S. acknowledges the financial support of Funda\c{c}\~ao para a Ci\^encia e a Tecnologia (FCT), Portuguese Agency for Scientific Research, through the project
UID/PRR/04621/2025 of CEMAT/IST-ID, DOI https://doi.org/10.54499/UID/PRR/04621/2025.}

\date{}

\maketitle


\centerline{$^1$ Institute of Mathematics, Czech Academy of Sciences,}

\centerline{\v Zitn\'a 25, 115 67 Praha 1, Czech Republic.}

\centerline{$^2$ BCAM, Basque Center for Applied Mathematics}

\centerline{Mazarredo 14, E48009 Bilbao, Bizkaia, Spain.}

\centerline{$^3$IKERBASQUE, Basque Foundation for Science, }

\centerline{Plaza Euskadi 5, 48009 Bilbao, Bizkaia, Spain.}

\centerline{$^4$ Department of Mathematics and CEMAT, Instituto Superior T\'{e}cnico,}

\centerline{1049-001 Lisboa, Portugal.}

\medskip

\begin{abstract}
We investigate the steady self-propelled motion of a rigid body immersed in a three-dimensional incompressible viscous fluid governed by the Navier-Stokes equations. The analysis is performed in a body-fixed reference frame, so that the fluid occupies an exterior domain and the propulsion mechanism is modeled through nonhomogeneous Navier-slip boundary conditions at the fluid-body interface. Such conditions provide a realistic description of propulsion in microfluidic and rough-surface regimes, where partial slip effects are significant.
Under suitable smallness assumptions on the boundary flux and on the normal component of the prescribed surface velocity, we establish the existence of weak steady solutions to the coupled fluid-structure system. A key analytical ingredient is the derivation of a Korn-type inequality adapted to exterior domains with rigid-body motion and Navier-slip interfaces, which yields uniform control of both the fluid velocity and the translational and rotational velocities of the body. Beyond existence, we provide a necessary and sufficient condition under which a prescribed slip velocity on the body surface induces nontrivial translational or rotational motion of the rigid body. This is achieved through the introduction of a finite-dimensional thrust space, defined via auxiliary exterior Stokes problems with Navier boundary conditions, which captures the effective contribution of boundary-driven flows to the rigid-body motion. Our results clarify how boundary effects generate propulsion and extend the classical Dirichlet-based theory to the Navier-slip setting.

\end{abstract}


\tableofcontents


\section{Introduction}
The interaction between viscous fluids and immersed rigid bodies represents one of the classical problems in mathematical fluid mechanics. The coupled dynamics of the body and the surrounding fluid, usually referred to as fluid–structure interaction (FSI), are governed by the Navier–Stokes equations coupled with the Newton–Euler equations for the rigid body. Such systems arise naturally in many physical contexts, from the sedimentation of particles to the motion of biological swimmers and artificial micro-robots. The mathematical analysis of these problems combines nonlinear partial differential equations with moving boundary dynamics and has been the focus of extensive research in recent decades (see, e.g., \cite{FSIforBIO,ParticlesInFlow,GaRev}).

Most classical studies consider passive motion, where the body moves solely in response to external forces such as gravity or fluid drag. In contrast, self-propelled motion refers to situations where the body generates its own propulsion by creating localized flow disturbances near its boundary. This mechanism underlies the locomotion of both macroscopic organisms (birds, fish), or microorganisms such as ciliates and flagellates as well as engineered self-propelling particles. From a mathematical perspective, such systems are modeled by imposing nonhomogeneous boundary conditions on the fluid velocity, which represent the thrust generated by surface activity in the absence of external forces. 

A pioneering analysis of steady self-propelled motion in an incompressible viscous fluid was carried out by Galdi \cite{Galdi1999}, who established the existence of steady weak solutions under Dirichlet boundary conditions. Subsequent works, including Silvestre \cite{Sil1,MR1953783}, refined the understanding of attainability and symmetry of such motions. These studies demonstrated that steady self-propulsion can emerge as a balance between hydrodynamic drag and boundary-induced thrust. However, all of these investigations relied on the no-slip assumption, where the fluid velocity coincides with that of the body along the boundary. While this condition is mathematically convenient and physically accurate for many smooth surfaces, it becomes inadequate in situations involving rough or micro-structured boundaries, near-contact flows, or partial slip regimes, as discussed in \cite{GVHil2,GVHil3,Bucur2010,Masmoudi2010,JagerMikelic2001, NRRS}. In \cite{mmnp}, the authors established the global existence of weak solutions describing the self-propelled motion of a rigid body in a density-dependent incompressible viscous fluid under Navier-type slip conditions. To the best of our knowledge, no previous work addresses steady self-propelled rigid-body motion in exterior Navier-Stokes flows under Navier-slip boundary conditions.

In the present paper, we investigate the steady self-propelled motion of a rigid body in a viscous incompressible fluid governed by the three-dimensional Navier–Stokes equations, subject to Navier slip boundary conditions. The nonhomogeneous Navier boundary data represent the thrust generated along the body surface. In this setting, the absence of external forces requires the total hydrodynamic force and torque acting on the body to vanish, which couples the boundary-driven fluid motion to the translational and rotational velocities of the body. Our main result establishes the existence of weak steady solutions to the coupled system under the assumption of sufficiently small boundary flux and small normal component of the prescribed thrust relative to the flux carrier. The results presented here extend the Dirichlet-based self-propulsion theory to the Navier-slip setting, thereby bridging two distinct modeling frameworks for viscous locomotion. Beyond their mathematical interest, Navier-slip formulations provide a more realistic description of propulsion in microfluidic and rough-surface contexts, where partial slip effects naturally arise and cannot be neglected.

The analysis of the steady self-propelled motion of a rigid body immersed in an incompressible viscous fluid under Navier-slip boundary conditions in an exterior domain presents a number of conceptual and analytical challenges. Our work develops a complete framework to handle these difficulties, leading to new existence results and a deeper understanding of the propulsion mechanism in viscous fluids. Below, we summarize the main contributions, highlight the essential difficulties, and describe the key strategies that make the analysis possible.

The first major contribution of this paper is the introduction of a Sobolev space framework suitable for the analysis of self-propelled motion in unbounded domains. Classical Sobolev spaces are inadequate for controlling the behavior of the velocity field at infinity, particularly in the presence of nontrivial rigid-body motion. The weighted setting we introduce allows one to capture the decay properties of the fluid. This setting is essential for establishing uniform estimates and passing to the limit in the invading domain approximation. Unlike the classical no-slip case, the Navier-slip condition introduces a discontinuity in the tangential component of the velocity field at the interface between the fluid and the rigid body. This creates profound analytical difficulties, as the weak formulation now contains boundary integrals involving the tangential velocity. Properly handling these terms requires a delicate balance between the regularity of the test functions, the trace properties of the velocity, and the friction term on the boundary. We develop a consistent weak formulation that incorporates these effects and prove the necessary integration-by-parts and Korn-type inequalities adapted to this setting.

A central analytical challenge in exterior domain problems is to construct solutions by means of the invading domain technique, in which one solves the problem in a sequence of expanding bounded domains and passes to the limit as the radius tends to infinity.
A crucial component of this method in the context of our article is the derivation of a Korn-type inequality uniform with respect to the domain size, which allows control of both the translational and rotational components of the rigid body’s velocity in terms of the fluid’s deformation tensor. We prove such an inequality and, in particular, establish an estimate linking the translational and angular velocities of the rigid body to the appropriate fluid velocity norms (see Lemma \ref{ledesfnl}). This result is fundamental for obtaining uniform energy estimates and compactness in the limit process. An essential and technically subtle step in the analysis is the extension of the normal component of the self-propelling velocity from the body boundary into the fluid domain. This extension is \textit{highly nonstandard} due to the coupling between the normal and tangential components introduced by the Navier-slip law. We construct such an extension with control of its divergence and boundary behavior (see Lemma \ref{lemaext}). The smallness assumption appearing in our main theorem (Theorem \ref{Mainresult3}) originates precisely from the estimates obtained in this construction and reflects the delicate balance between boundary-driven thrust and viscous dissipation inherent to steady self-propelled regimes.

The proof of existence for the Navier–Stokes system relies on a carefully designed Galerkin approximation adapted to the Navier-slip boundary conditions. Constructing an appropriate divergence-free basis satisfying the slip condition is itself nontrivial and constitutes one of the technical achievements of this paper.
We then establish existence in bounded domains combining uniform estimates of the nonlinear term with compactness arguments. Passing to the limit in the invading domain procedure requires uniform control of the nonlinear convective term, which forms the core of the analytical difficulty and the heart of the proof.

Beyond existence theory, we address the fundamental question of when a prescribed self-propelling surface velocity actually induces motion of the rigid body. We show that the rigid body moves if and only if the projection of the self-propelling velocity onto a certain finite-dimensional space, which we call the thrust space, is non-zero.
To establish this characterization, we study several auxiliary exterior Stokes problems with Navier boundary conditions and construct an appropriate lifting operator (see \eqref{liftv*}). These auxiliary fields describe the effective interaction between the fluid and the body under surface actuation and form the analytical foundation of the thrust space concept. We first establish the relation between the self-propelling velocity and the resulting rigid-body velocity in the linear Stokes regime, where the analysis is more tractable. We then introduce and explicitly characterize the finite-dimensional basis of the thrust space under Navier-slip conditions which is a new and significant contribution of this work.
Finally, we extend these ideas to the nonlinear Navier-Stokes system, using refined cutoff techniques and uniform estimates to identify conditions under which genuine self-propulsion occurs. This allows us to rigorously bridge the gap between the Stokes and Navier-Stokes regimes (Theorem \ref{Mainresult2} and Theorem \ref{selfprop-NS}).

The paper is organized as follows. In Section \ref{sec1} we formulate the model for the steady motion of a self-propelled rigid body in a viscous incompressible fluid and recall the governing equations in the body-fixed frame. Section \ref{sec2} introduces the functional setting in weighted Sobolev spaces and states the main existence and characterization results on non-zero self-propelling condition. Section \ref{sec3} develops the analytical framework for the Navier-Stokes system: we construct a suitable divergence-free basis satisfying the Navier-slip condition, extend the normal component of the boundary data, and employ the invading domain technique together with uniform Korn-type estimates to prove the existence of weak solutions. Section \ref{sec4} concerns the analysis of propulsion by studying auxiliary exterior Stokes problems and introducing the finite-dimensional thrust space, we identify the conditions under which the rigid body acquires a nonzero velocity and relate the nonlinear Navier-Stokes system to its linear counterpart.

\section{Description of the model}\label{sec1}
Let $\mc{B}(t) \subset \mathbb{R}^3$ denote a closed, bounded and simply connected  rigid body. We assume that the rest of the space, i.e., 
$\mathbb{R}^{3} \setminus \mc{B}(t)=\Omega(t)$ is filled with a viscous incompressible nonhomogeneous fluid. 
The initial domain of the rigid body is denoted by $\mc{B}_0:=\mc{B}(0)$ and is assumed to have a smooth boundary. Correspondingly, $\Omega :=\mathbb{R}^{3} \setminus \mc{B}_0$ is  the initial fluid domain.

Our system of a rigid body moving in a fluid is in the first instance a moving domain problem in an inertial frame in which the velocity of the rigid body is described by
\begin{equation} \label{def-U_S} 
U_{\mc{S}}(y,t)= h'(t)+  r(t) \times (y-h(t)) \; \mbox{for} \quad (y,t) \in \mc{B}(t)\times (0,T). \end{equation} 
Here, $h'(t)$ is the linear velocity of the centre of mass $h(t)$ and $r(t)$ is the angular velocity of the body. 
The solid domain at time $t$ in the same inertial frame is given by
\begin{equation*} \mc{B}(t)= \left\{h(t)+ Q(t)x \mid x\in \Omega\right\}, \end{equation*} 
where $Q(t) \in SO(3)$ is associated to the rotation of the rigid body.  Mathematically, it will turn out useful to define $U_{\mc{S}}$ for all $(y,t)\in \mathbb{R}^3 \times (0,T)$. 
Let $\varrho$, $U$ and $P$ be the density, velocity and pressure of the viscous
incompressible nonhomogeneous fluid, respectively, which satisfy 
\begin{align}  \label{mass:fluid}  \frac{\partial U}{\partial t} + \nabla \cdot\left( U
\otimes U\right) + \nabla P = \nabla \cdot(2\nu \mathbb{D}(U)), \quad 
 \mathrm{div}\,U=0 \quad \mbox{ in } \Omega(t) \times (0,T).
\end{align}
 Here we consider a self-propelled motion of the body $\mc{B}(t)$,  which we describe by a vectorial flux $V_{*}$ at $\partial \Omega(t)$. In this article we consider Navier type conditions at the fluid-structure interface; we prescribe the  normal and tangential parts of the flux by: 
\begin{align*} 
V_{*}\cdot N &= (U-U_{\mc{S}})\cdot N \quad \mbox{ on } \partial \Omega(t) \times (0,T), \\
  (V_{*} \times N) &= \lambda (2\mathbb{D}(U)N)\times N +  
(U-U_{\mc{S}})\times N  \quad \mbox{ on } \partial \Omega(t)\times (0,T), 
\end{align*}
where $N$ is the unit outward normal to the boundary of $\Omega(t)$, i.e., directed towards $\mc{B}(t)$ and $\lambda$ is the slip length which has the unit of a length. The quantity $\lambda$ can be interpreted as the fictitious distance below the surface where the no-slip boundary condition would be satisfied

 For a given viscosity coefficient $\nu >0$, we set $$\Sigma (U,P) = -P\mathbb{I} + 2\nu \mathbb{D}(U) \quad \mbox{ with } \quad \mathbb{D}(U)= \left[\frac{1}{2}\left(\frac{\partial U_i}{\partial x_j} + \frac{\partial U_j}{\partial x_i}\right)\right]_{i,j=1,2,3}.$$
Let $\varrho_{\mc{B}}$ denote the density of the rigid body. Then $m = \int_{\mc{B}(t)} \varrho_{\mc{B}}(y,t) dy$ is its total mass, which is constant in time and does not change under the coordinate transformation below. Further, 
the moment of inertia $J(t)$ is defined by
$$J(t) =  \int\limits_{\mc{B}(t)} \varrho_{\mc{B}}(y,t) 
 \Big(|y-h(t)|^2\mathbb{I} - (y-h(t))\otimes (y-h(t))\Big)\, dy. $$
The ordinary differential equations modeling the dynamics of the rigid body then read 
\begin{align*} mh''&= -\int\limits_{\partial \Omega(t)} \left(\Sigma (U,P) N - U(U-U_{\mc{S}})\cdot N\right)\, dS , 
\\ 
(Jr)' &= -\int\limits_{\partial \Omega(t)} (x-h) \times \left(\Sigma (U,P) N - U(U-U_{\mc{S}})\cdot N\right)\, dS,
\end{align*} 
We suppose the velocity of the fluid satisfy:
\begin{equation}\label{13:10}
 U(y,t) \rightarrow 0 \mbox{ as }|y| \rightarrow \infty, 
\end{equation}

as well as the initial conditions  \begin{equation}\label{initialcond}  U(y,0)=u_0(y),\, h(0)=0,\, h'(0)=\xi_{0},\, r(0)=\omega_{0}. 
\end{equation}
We apply a global change of variables that transforms the system in such a way that it is formulated in a frame which is attached to the rigid body. At time $t=0$, the two frames are assumed to coincide with $h(0) =0$. Hence the transformed system is posed on the fixed domain  $\Omega \times (0,T)$ via the following change of variables: 
\begin{align*} 
v(x,t) =Q(t)^{\top}U(Q(t)x+h(t),t), \quad v_{*}(y,t) = Q(t)^{\top} V_{*}(Q(t)y+h(t),t), \quad p(y,t)=P(Q(t)y+h(t),t).
\end{align*} 
Further, the moment of inertia transforms to 
$$ J_0 = Q(t)^\top J(t) Q(t).$$
By \eqref{def-U_S}
and the extension of $U_{\mc{S}}$ to the whole space, we have that  \color{black}
\begin{equation*} 
V(x,t) = Q(t)^\top U_{\mc{S}}(Q(t)y + h(t),t)  =  \xi(t) + \omega(t) \times x, \quad x\in \Omega \times (0,T),
\end{equation*}
where $\xi(t) = Q(t)^{\top}h'(t)$ is the transformed linear velocity and $\omega(t)=Q(t)^{\top}r(t)$ the transformed angular velocity of the rigid body. The normal and tangential parts of the self-propulsion flux in the new frame are given by: 
\begin{align} 
v_{*}\cdot n &= 0 \mbox{ on } \partial \Omega \times (0,T), \label{self-propel normal}\\
  (v_{*} \times n) &= \lambda (D(u)n)\times n +  
(u-V)\times n  \mbox{ on } \partial \Omega \times (0,T) \label{self-propel tangent}
\end{align}
with $n(x,t) = Q(t)^\top N(Q(t)x + h(t))$ for $(y,t) \in \partial \Omega \times (0,T)$ denoting the inward pointing normal to $\partial \Omega$.
With all this at hand, equations \eqref{mass:fluid}--\eqref{initialcond} can be rewritten in the fixed domain as
\begin{equation}
\label{chg of var momentum:fluid}
\left\{
    \begin{aligned}
 \frac{\partial v}{\partial t} + \nabla \cdot\left[v\otimes (v-V)\right] + \omega \times v + \nabla p &= \nabla \cdot(2\nu \mathbb{D}(v)), \quad \mathrm{div}\,v=0 \quad \mbox{ in } \Omega \times (0,T),\\ (v-V)\cdot n &= v_{*} \cdot n \mbox{ on } \partial \Omega\times (0,T), \\ \lambda( 2\mathbb{D}(v)n)\times n + (v-V)\times n &= v_{*}\times n \mbox{ on } \partial \Omega\times (0,T), \\ 
m\xi' + m\omega \times \xi &= -\int\limits_{\partial \Omega} (\mathbb{T}_{\nu} (v,p) n - v(v-V)\cdot n)\, dS , \\ J_{0}\omega' + \omega \times J_{0}\omega &= -\int\limits_{\partial \Omega} x \times (\mathbb{T}_{\nu} (v,p) n-v(v-V)\cdot n)\, dS  , \\  v(x,0)=u_0(x),\, h(0)&=0,\, \xi(0)=\xi_{0},\, \omega(0)=\omega_{0},  \end{aligned}
\right.
\end{equation}
where $V(x)=\xi + \omega \times x$ and
\begin{align*}
 \mathbb{T}_{\nu} (v,p) = -p\mathbb{I} + 2\nu \mathbb{D}(v)\mbox{ with  }\mathbb{D}(v)&= \left[\frac{1}{2}\left(\frac{\partial v_i}{\partial x_j} + \frac{\partial v_j}{\partial x_i}\right)\right]_{i,j=1,2,3}. 
\end{align*}
 It follows from \eqref{13:10} that $v(x,t)$ satisfies 
$v(x,t) \rightarrow 0 \mbox{ as }|x| \rightarrow \infty.$

We can rewrite the system \eqref{chg of var momentum:fluid} in the following non-dimensional form:
\begin{equation}
\label{momentum:fluid-nondim}
\left\{
    \begin{aligned}
 \mathrm{St} \, \mathrm{Re}\frac{\partial v}{\partial t} + \mathrm{Re}\left(\nabla \cdot\left[v\otimes (v-V)\right] + \omega \times v \right)+ \nabla p &= \Delta v, \quad \mathrm{div}\,v=0 \quad \mbox{ in } \Omega \times (0,T),\\ (v-V)\cdot n &= v_{*} \cdot n \mbox{ on } \partial \Omega\times (0,T), \\  2{\mathbb D}(v)n\times n +\alpha (v-V)\times n &= \alpha v_{*}\times n \mbox{ on } \partial \Omega\times (0,T), \\ 
\mathrm{St} \mathrm{Re} m\xi' + \mathrm{Re}\ m\omega \times \xi &= -\int\limits_{\partial \Omega} (\mathbb{T} (v,p) n - \mathrm{Re}\ v(v-V)\cdot n)\, dS , \\ \mathrm{St} \mathrm{Re} J_{0}\omega' + \mathrm{Re}\ \omega \times J_{0}\omega &= -\int\limits_{\partial \Omega} x \times (\mathbb{T} (v,p) n- \mathrm{Re}\ v(v-V)\cdot n)\, dS  , \\  u(x,0)=u_0(x),\, h(0)&=0,\, \xi(0)=\xi_{0},\, \omega(0)=\omega_{0},  \end{aligned}
\right.
\end{equation}
where 
\begin{align*}
 \mathbb{T} (v,p) = -p\mathbb{I} + 2 \mathbb{D}(v)\mbox{ with  }\mathbb{D}(v)&= \left[\frac{1}{2}\left(\frac{\partial v_i}{\partial x_j} + \frac{\partial v_j}{\partial x_i}\right)\right]_{i,j=1,2,3}. 
\end{align*}
In the above system \eqref{momentum:fluid-nondim}, all the variables are non-dimensional. Here $\mathrm{Re}=\frac{WL}{\nu}$ is the Reynolds number and $\mathrm{St}=\frac{L}{\tau W}$ is the Strouhal number where $W, L$, $\tau$ are the suitable scales for the velocity, length and time. The dimenionless quantity $\alpha = \frac{\lambda}{L}$ is the friction coefficient which measures the tendency of the fluid to slip on the boundary. 

When $\mathrm{St}\rightarrow 0$ (if $\tau$ is very large), we obtain the \textit{steady version} of the  nonlinear problem \eqref{momentum:fluid-nondim}: 
\begin{equation}
\label{eqn1:NS}
\left\{
    \begin{aligned} -\displaystyle \nabla \cdot  {\mathbb T}(v,p)+\mathrm{Re}\left(\nabla \cdot\left[v\otimes (v-V)\right] + \omega \times v \right) = 0 \quad \mbox{ in } \Omega, \\
    \nabla \cdot v  = 0 \quad \mbox{ in } \Omega ,\\ (v-V)\cdot n = v_{*} \cdot n \quad \mbox{ on } \partial \Omega, \\  2[{\mathbb D}(v) n]\times n +\alpha (v-V)\times n = \alpha v_{*}\times n \quad \mbox{ on } \partial \Omega, \\ \displaystyle \lim_{\left| x\right| \rightarrow
\infty }v(x) =0, \\ 
 \mathrm{Re}\ m\omega \times \xi  = -\int\limits_{\partial \Omega} (\mathbb{T} (v,p) n - \mathrm{Re}\ v(v-V)\cdot n)\dS , \\ \mathrm{Re}\ \omega \times J_{0}\omega = -\int\limits_{\partial \Omega} x \times (\mathbb{T} (v,p) n- \mathrm{Re}\ v(v-V)\cdot n)\dS  .  \end{aligned}
\right.
\end{equation}

In the absence of external forces, the forward force (thrust) that makes $\mathcal B$ move is generated by $\mathcal B$ itself. The motion is due to the interaction of the body's external surface and the fluid via the boundary values $v_*$.

\section{Main results}\label{sec2}
In order to state the main contributions of this article, we need to introduce some notations.
\subsection{Notation}
Throughout the paper, we will  use the same font style to represent scalar, vector, and tensor-valued functions.  For any set $\mathcal{A}$, we consider  ${\mathcal D}(\mathcal{A}) = C^\infty_0(\mathcal{A})$. Standard notations $L^p(\mathcal O)$, $W^{m,p}(\mathcal O)$ and $H^m(\mathcal O):=W^{m,2}(\mathcal O)$, for any domain $\mathcal O$, with $p\in [1,\infty]$ and $m \in {\mathbb N} \cup \{0\}$,
will be adopted for Lebesgue and Sobolev spaces. For corresponding norms, we write
$\|{\cdot}\|_{p,\mathcal{O}}\coloneqq\|{\cdot}\|_{L^p(\mathcal O)}$
and $\|{\cdot}\|_{m,p,\mathcal O}\coloneqq\|{\cdot}\|_{W^{m,p}(\mathcal O)}$. Concerning boundary traces, we set usual notations $L^p(\Gamma)$, $W^{m-\frac{1}{p},p}(\Gamma)$, $H^{m-\frac{1}{2}}(\Gamma)$ with the corresponding norms $\|{\cdot}\|_{p,\Gamma}$, $\|{\cdot}\|_{m-\frac{1}{p},p,\Gamma}$ and $\|{\cdot}\|_{m-\frac{1}{2},2,\Gamma}$ for a boundary $\Gamma$ of a domain.

Let $\rho(x):=(1+|x|^2)^{1/2}$ and consider the following weighted Sobolev spaces and norms:
\begin{equation*}
H^{m,k}_{\rho}(\Omega) = \left\{u\in {\mathcal D}'(\Omega) \mid  \rho^{k-m+|\lambda|}D^{\lambda}u \in L^2(\Omega) \ \forall \ |\lambda|\leq m \right\},\quad \lambda\in (\mathbb{N}\cup \{0\})^3, \quad m \in {\mathbb N} \cup \{0\}, \quad k \in {\mathbb Z},
\end{equation*}
with the norm 
\begin{equation*}
\| u \|_{H^{m,k}_{\rho}(\Omega) } = \sum\limits_{|\lambda|\leq m}\left( \left\| \rho^{k-m+|\lambda|}D^{\lambda}u  \right\|^2_{2,\Omega}   \right)^{\frac 12}. 
\end{equation*}
In particular, we need the following spaces throughout the  paper:
\[
H^{0,-1}_{\rho}(\Omega) = \left\{u\in {\mathcal D}'(\Omega) \mid  \frac{u}{\rho} \in L^2(\Omega)\right\}, \quad \| u \|_{H^{0,-1}_{\rho}(\Omega) } = \left\| \frac{u}{\rho}  \right\|_{2,\Omega}
\]
\[
H^{1,0}_{\rho}(\Omega) = \left\{u\in {\mathcal D}'(\Omega) \mid \frac{u}{\rho} \in L^2(\Omega),\, \nabla u \in L^2(\Omega)\right\}, \quad \| u \|_{H^{1,0}_{\rho}(\Omega) } = \left( \left\| \frac{u}{\rho}  \right\|^2_{2,\Omega} + \left\| \nabla u \right\|^2_{2,\Omega} \right)^{\frac 12} 
\]
\[
H^{2,1}_{\rho}(\Omega) = \left\{u\in {\mathcal D}'(\Omega) \mid \frac{u}{\rho} \in L^2(\Omega),\, \nabla u \in L^2(\Omega), \, \rho D^2 u \in L^2(\Omega)\right\}, 
\]
\[
\| u \|_{H^{2,1}_{\rho}(\Omega) } = \left( \left\| \frac{u}{\rho}  \right\|^2_{2,\Omega} + \left\| \nabla u \right\|^2_{2,\Omega} + \left\| \rho D^2 u \right\|^2_{2,\Omega}  \right)^{\frac 12} .
\]
The Hardy inequality holds for $H^{1,0}_{\rho}(\Omega)$,
\begin{equation*}
\forall u \in H_{\rho}^{1,0}(\Omega),\quad
\| u \|_{H_{\rho}^{1,0}(\Omega)}\leq C\|\nabla u\|_{2,\Omega}.
\end{equation*}
The space ${\mathcal D}(\overline{\Omega})$ is dense in $H^{m,k}_{\rho}(\Omega)$.

Throughout the paper, we will use the Einstein summation convention and implicitly sum over all repeated indices. 
\subsection{Main results}
Our main contribution in this paper is that we are able to prove the existence of weak solutions for system \eqref{eqn1:NS} that describes the fully nonlinear steady self-propelled motion with Navier-slip boundary condition. Let us introduce the set of rigid velocity fields: 
\begin{equation*}
\mathcal{R}=\left\{ V : \mathbb{R}^3 \to \mathbb{R}^3 \mid \exists\ \xi, \omega \in \mathbb{R}^3 \mbox{ such that }V(x)=\xi+ \omega \times x\mbox{ for any } x\in\mathbb{R}^3\right\}.
\end{equation*}
\begin{theorem}\label{Mainresult3}
Let $\partial\Omega $ be of class $C^{2,1}$ and  $v_{*}\cdot n \in H^{1/2}(\partial\Omega)$, $v_{*}\times n \in H^{1/2}(\partial\Omega)^3$. Let us set 
\begin{equation*}
\sigma (x):= \frac{x}{4 \pi |x|^3}, \qquad \Phi = \int\limits_{\partial\Omega}  v_* \cdot n\ dS \not=0.
\end{equation*}
There exist $C_1 (\mc{B}_0), C_2 (\mc{B}_0)$ such that if 
\begin{equation*}
\mathrm{Re}|\Phi|< C_1, \quad \mathrm{Re}\| (v_{*}\cdot n )n  -  \Phi \sigma|_{\partial \Omega}\|_{H^{1/2}(\partial\Omega)} < C_2,
\end{equation*}
then the problem \eqref{eqn1:NS} admits at least one weak solution $(v,V)\in H^{1,0}_{\rho}(\Omega) \times \mathcal{R}$. Moreover, there exists $C=C(\Omega)>0$ such that
{
\begin{equation}\label{est:cont}
\|\mathbb{D}(v)\|_{2,\Omega} + \alpha \| [v-v_{\mathcal S}]_{\tau} \|_{2,\partial\Omega}\leq C(\Omega) (\| v_* \cdot n \|_{1,2,\Omega}  + {\mathrm{Re}} \,  \| v_* \cdot n \|_{1,2,\Omega}^2)  + \alpha \| [v_*]_{\tau} \|_{2,\partial \Omega}.
\end{equation}
}
\end{theorem}
If $\mathrm{St}\rightarrow 0$ (if $\tau$ is very large) and $\mathrm{Re}\rightarrow 0$ (if the velocity of the rigid body $\mathcal{B}(t)$ is very small or if the viscosity of the fluid is very large), then we obtain the following \textit{steady Stokes approximation}: 
\begin{equation}
\label{eqn1tan}
\left\{
    \begin{aligned}
 \nabla \cdot  {\mathbb T}(v,p)= 0, \quad  \nabla \cdot  v = 0 
  \quad \mbox{ in }\Omega,  \\  (v  -V)\cdot n
= v_{*} \cdot n \quad  \mbox{ on }\partial\Omega, 
 \\ 2 [{\mathbb D}(v) n] \times n + \alpha (v-V) \times n 
= \alpha v_{*} \times n  \quad \mbox{ on }\partial\Omega,
 \\  \lim_{\left| x\right| \rightarrow
\infty }v(x) = 0, \\
\displaystyle \int\limits_{\partial\Omega} {\mathbb T}(v,p)n \dS  = 0, \quad  \int\limits_{\partial\Omega}  x \times {\mathbb T}(v,p)n \dS  = 0.
 \end{aligned}
\right.
\end{equation}
 We have the following wellposedness result for the system \eqref{eqn1tan}:
\begin{theorem}\label{Mainresult1} 
 Let $\partial\Omega $ be of class $C^{2,1}$ and  $v_{*}\cdot n \in H^{3/2}(\partial\Omega)$, $v_{*}\times n \in H^{1/2}(\partial\Omega)^3$. Then problem \eqref{eqn1tan} admits one and only one solution $(v,p, V)$ such that  
\begin{equation*}
v \in  H^{2,1}_{\rho}(\Omega)^3, \quad
p \in H^{1,1}_{\rho} (\Omega),\quad V:=\xi + \omega \times x\in\mathcal{R}.
\end{equation*}
 The following estimates hold  
\begin{equation*}
\|(1+|x|^2)^{-\frac{1}{2}}v\|_{2,\Omega}+\| \nabla v \|_{1, 2, \Omega} + \| p \|_{1,2,\Omega} + |\xi  |+|\omega | \leq C \left(\|v_{*}\cdot n\|_{3/2,2,\partial\Omega}+ \|v_{*}\times n\|_{1/2,2,\partial\Omega}\right).
\end{equation*}
\end{theorem}
In order to solve the linear problem \eqref{eqn1tan}, we will need a number of auxiliary classical exterior Stokes problems with Navier boundary conditions. The standard basis of ${\mathbb R}^3$ will be denoted by $\{\textsf{e}_1,\textsf{e}_2,\textsf{e}_3\}$. For each $i=1,...,6$, we define the elementary rigid motion
\begin{equation}\label{elrigid}
\widetilde{\textsf{e}_i}=\left\{ 
\begin{array}{lcl}
\textsf{e}_i & \mbox{if} & i=1,2,3, \medskip \\ 
\textsf{e}_{i-3}\times x & \mbox{if} & i=4,5,6,
\end{array}
\right. 
\end{equation} 
and consider $(H^{(i)},P^{(i)})$, $i=1,2,\cdots 6$, as the solutions of the auxiliary Stokes problems:   
\begin{equation}
\label{astok}
\left\{
    \begin{aligned}
\nabla \cdot  {\mathbb T}(H^{(i)},P^{(i)})=0, \quad 
\nabla \cdot  H^{(i)}=0 
\quad  \mbox{ in } \Omega,  \\ 
\displaystyle H^{(i)} \cdot n
=  \widetilde{\textsf{e}_i} \cdot n \quad \mbox{ on }\partial\Omega, \\ \displaystyle 2 [{\mathbb D}(H^{(i)}) n] \times n + \alpha H^{(i)} \times n 
= \alpha \widetilde{\textsf{e}_i} \times n \quad \mbox{ on }\partial\Omega,  \\ 
\lim_{|x|\rightarrow \infty }H^{(i)}(x)=0.
\end{aligned}
\right.
\end{equation}
Similarly to the linear problem studied in \cite{Galdi1999,GaRev}, we will see in Section \ref{Aux:rigid} that solving the Stokes approximation \eqref{eqn1tan} is equivalent to solving a system of linear equations involving the matrix $M \in {\mathbb R}^{6\times 6}$ which is defined as
\begin{equation}
M_{ij} = \int\limits_{\partial\Omega} \widetilde{\textsf{e}_i}  \cdot  {\mathbb T}(H^{(j)},P^{(j)})  n\dS, \quad i,j=1,...,6.
\label{mat}
\end{equation}
We can rewrite $M$ in the form
$$
M = \bigg[ \begin{array}{cc}K & S^\top \\ S & R  \end{array}\bigg],
$$
where $K,R,S \in {\mathbb R}^{3\times 3}$ are given by

\begin{empheq}[right=\empheqrbrace]{align}
K_{ij} &= M_{ij},\quad  i,j=1,2,3 \notag\\
R_{(i-3)(j-3)} &= M_{ij},\quad  i,j=4,5,6 \label{def:KRS} \\
S_{(i-3)j} &= M_{ij},\quad  i=4,5,6, \, j=1,2,3. \notag
\end{empheq}

Another problem we are interested in  is the following: in which way one should prescribe boundary velocities that can propel the rigid body with a \emph{nonzero} velocity $V$? Our aim is to establish that the rigid body moves if and only if ${\mathbb P}(v_{*})\not= 0$, where ${\mathbb P}$ is the projection operator on the \emph{thrust space} ${\mathcal{T}}({\mathcal B}_0)$  defined in \eqref{thrust}.
\begin{theorem}\label{Mainresult2}
Let $\partial\Omega$ be of class $C^{2,1}$. Then for any $v_{*}$ satisfying $v_{*}\cdot n \in H^{3/2}(\partial\Omega)$, $v_{*}\times n \in H^{1/2}(\partial\Omega)^3$ and $\mathbb{P}(v_{*})\neq 0$, there exists a unique solution $(v,p, V)$ to the problem \eqref{eqn1tan} such that $V \in \mathcal{R} \setminus 0$. Conversely, for any $V:=\xi + \omega \times x \in \mathcal{R} \setminus 0$, there exists one and only one solution $(v,p)$ to the problem \eqref{eqn1tan} such that $v_{*}$ belongs to ${\mathcal{T}}({\mathcal{B}})$. Moreover, the translational velocity $\xi$ of the rigid body is nonzero iff 
\begin{equation*}
(W_1,W_2,W_3)\neq S^{\top}R^{-1}(W_4,W_5,W_6),
\end{equation*}
and its angular velocity $\omega$ is nonzero iff
\begin{equation*}
(W_4,W_5,W_6)\neq S K^{-1}(W_1,W_2,W_3).
\end{equation*}
Here $K,S,R$ are defined in \eqref{def:KRS} and 
\begin{equation*}
W_i = - \int\limits_{\partial\Omega} v_{*}\cdot {\mathbb T}(H^{(i)},P^{(i)})  n\dS,\quad i=1,2,\cdots 6. 
\end{equation*}
\end{theorem}
In the case of the motion of a body in the Navier-Stokes liquid, we prove that  self-propelled motion occurs, i.e, $V \neq 0$, whenever there is a nonzero orthogonal projection of $v_{*}$ on the \emph{thrust space} ${\mathcal{T}}({\mathcal{B}})$ and $\mathrm{Re}$ is not ``too large''. More precisely, we have

\begin{theorem}\label{selfprop-NS}
Let $\partial\Omega $ be of class $C^{2,1}$ and  $v_{*}\cdot n \in H^{1/2}(\partial\Omega)$, $v_{*}\times n \in H^{1/2}(\partial\Omega)^3$. Let us set 
\begin{equation*}
\sigma (x):= \frac{x}{4 \pi |x|^3}, \qquad \Phi = \int\limits_{\partial\Omega}  v_* \cdot n\ dS \not=0.
\end{equation*}
Assume that there exist $C_1 (\mc{B}_0), C_2 (\mc{B}_0)$ such that 
\begin{equation*}
\mathrm{Re}|\Phi|< C_1, \quad \mathrm{Re}\| (v_{*}\cdot n )n  -  \Phi \sigma|_{\partial \Omega}\|_{H^{1/2}(\partial\Omega)} < C_2.
\end{equation*}
Suppose $\mathbb{P}(v_{*})\neq 0$, where ${\mathbb P}$ is the projection operator on the \emph{thrust space} ${\mathcal{T}}({\mathcal{B}_0}
)$ defined in \eqref{thrust}. Then there exists $C=C(v_*)>0$ such that if $\mathrm{Re} < C$, then the problem \eqref{eqn1:NS} admits at least one weak solution $(v,V)\in H^{1,0}_{\rho}(\Omega) \times \mathcal{R}$ with $V :=\xi+\omega\times x \neq 0$. In particular, if $(W_1,W_2,W_3)\neq S^{\top}R^{-1}(W_4,W_5,W_6)$, then 
\begin{equation*}
\frac{1}{2}|A\left((W_1,W_2,W_3) - S^{\top}R^{-1}(W_4,W_5,W_6)\right)| \leq |\xi| \leq \frac{3}{2}|A\left((W_1,W_2,W_3) - S^{\top}R^{-1}(W_4,W_5,W_6)\right)|,
\end{equation*}
while, if $(W_4,W_5,W_6)\neq S K^{-1}(W_1,W_2,W_3)$, then 
\begin{equation*}
\frac{1}{2}|B\left((W_4,W_5,W_6) - S K^{-1}(W_1,W_2,W_3)\right)| \leq |\omega| \leq \frac{3}{2}|B\left((W_4,W_5,W_6) - S K^{-1}(W_1,W_2,W_3)\right)|.
\end{equation*}
Here {$A=(K-S^{\top}R^{-1}S)^{-1}$, $B=(R-SK^{-1}S^{\top})^{-1}$}, $ K,S,R$ are defined in \eqref{def:KRS},  and 
\begin{equation*}
W_i = - \int\limits_{\partial\Omega} v_{*}\cdot {\mathbb T}(H^{(i)},P^{(i)})  n\dS,\quad i=1,2,\cdots 6. 
\end{equation*}
\end{theorem}
\begin{remark}
In exterior domains, it is not necessary to require the null flux condition $\int_{\partial\Omega} v_* \cdot n \dS=0$ for the Dirichlet boundary values. In the case of Navier boundary conditions, the same happens for the boundary values $v_*$. However, when $\Omega$ is a bounded domain, appropriate compatibility conditions for the data have to be imposed.

\end{remark}

\begin{remark}\label{BCv}
{A vector field $v=v(x)$ defined on $\partial\Omega$ can be written as
$$
v = (v \cdot n) n  + v_{\tau}, \quad\mathrm{where}\quad v_{\tau}:=   n \times (v \times n).
$$}
 We use the boundary conditions 
\begin{equation*}
 (v  -V)\cdot n
= v_{*} \cdot n \mbox{ on }\partial\Omega\quad \mbox{ and }\quad
 2 [{\mathbb D}(v) n] \times n + \alpha (v-V) \times n 
= \alpha v_{*} \times n \mbox{ on }\partial\Omega
\end{equation*}
to obtain
\begin{equation*}
 v  -V =  [(v  -V)\cdot n ] n + n \times [ (v-V) \times n )] =  v_{*}  - \frac{2}{\alpha} n \times ( [{\mathbb D}(v) n] \times n).
\end{equation*}
\end{remark}

\begin{remark}  We have the following identity regarding the scalar product of the tangential parts of two vectors:
$$
a_\tau \cdot b_\tau = [n \times (a \times n) ] \cdot [n \times (b \times n) ] =
$$
$$
= [a -  (a \cdot n) n ] \cdot [n \times (b \times n) ] = a \cdot [n \times (b \times n) ] = (a \times n ) \cdot  (b \times n). 
$$
\label{remtau}
\end{remark}

\section{Steady self-propelled motion in a Navier-Stokes fluid}\label{sec3}


Let us define the space of test functions
\[
{{\mathcal C} = {\mathcal C}(\Omega) }= \left\{ \varphi \in {\mathcal D}(\overline{\Omega})^3 \mid \nabla \cdot \varphi  = 0 \text{ in } \Omega, \, \exists \ \varphi_{\mathcal S}\in \mathcal{R} \mbox{ such that } \varphi \cdot n = \varphi_{\mathcal S} \cdot n \text{ on } \partial \Omega  \right\},
\]
where $\varphi_{\mathcal S} (x) = a_\varphi +  b_\varphi \times x$.
\begin{definition}
  We say that $(v,V) \in H^{1,0}_{\rho}(\Omega)^3 \times {\mathcal R}$ is a weak solution of system \eqref{eqn1:NS} if 
\begin{itemize}
\item $V(x)=\xi + \omega \times x$, 
\item $v$ is (weakly) divergence free in $\Omega$,
\item for {all $\varphi \in {\mathcal C}$},  
\begin{equation}\label{weakform}
\begin{aligned}
&2\int_{\Omega} {\mathbb D}({v}): {\mathbb D}({\varphi})\dx + \alpha \int_{\partial \Omega} [v-V]_{\tau}  \cdot [\varphi-\varphi_{\mathcal S}]_{\tau}\dS \\
=  & \,  {\mathrm{Re}} \int_{\Omega}[((v-V) \cdot \nabla ) \varphi ]\cdot v\dx - {\mathrm{Re}} \int_{\Omega} ( \omega \times v ) \cdot \varphi\dx 
 + {\mathrm{Re}} \, m (\xi\times\omega) \cdot a_\varphi  + {\mathrm{Re}}\, (J_0\omega\times\omega) \cdot b_\varphi  \\ 
 & - {\mathrm{Re}} \int_{\partial \Omega} (v_{*}\cdot n)[v]_{\tau}  \cdot [\varphi-\varphi_{\mathcal S}]_{\tau}\dS 
 +  \alpha \int_{\partial \Omega} [v_*]_{\tau}  \cdot [\varphi-\varphi_{\mathcal S}]_{\tau}\dS, 
 \end{aligned}
 \end{equation}
\item $(v-V) \cdot n = v_* \cdot n $ on $\partial \Omega$ (in the sense of trace).
\end{itemize}
\end{definition}
 Let us define
\begin{equation*} {\mathcal V} = {\mathcal V}(\Omega)= \{ u \in H^{1,0}_{\rho}(\Omega)^3  \mid   \nabla \cdot  u =0 \textrm{ in }\Omega, \exists  \, u_{\mathcal S}  \in {\mathcal R} \mbox{ such that } (u - u_{\mathcal S})\cdot n_{|\partial \Omega}= 0 \}
\end{equation*}
with corresponding scalar product 
\begin{equation*}
(u,v) \longmapsto 2 \int_{\Omega} {\mathbb D}(u): {\mathbb D}(v) \dx + \alpha \int_{\partial \Omega} [u-u_{\mathcal S}]_{\tau}  \cdot [v-v_{\mathcal S}]_{\tau} \dS.
\end{equation*}
Let
$B_R=\{ x \in {\mathbb R}^3 : |x| <R \}$ be the Euclidean ball of radius $R >0$,  along with the exterior domain $B^R=\{ x \in {\mathbb R}^3 : |x| > R \}$, and the spherical shell $B_{R_1,R_2} =\{ x \in {\mathbb R}^3 : R_1 < |x| < R_2 \}$. As the proof of the existence result (Theorem \ref{Mainresult3}) is based on the invading domains technique, we also define some function spaces in the bounded domains $\Omega_R=\Omega\cap B_R.$ We also set the notation $\Omega^R=\Omega\cap B^R$. {For $R >  \delta({\mathcal B}_0)$, we set (we use $n$ to denote the outward unit normal on $\partial \Omega$ and on $\partial B_R$)}
\begin{equation*}
{\mathcal{H}}_R = \{ u \in L^2(\Omega_R)^3 \mid \nabla \cdot  u =0 \textrm{ in }\Omega_R,\  \exists\ u_{\mathcal S} \in {\mathcal R} \mbox{ such that } (u - u_{\mathcal S})\cdot n_{|\partial\Omega}= 0 \text{ and } u\cdot n_{|\partial B_R}= 0 \}
\end{equation*}
with scalar product 
$$
(u,v) \longmapsto \ \int_{\Omega_R} u \cdot v \dx + m u_1 \cdot v_1 + u_2 \cdot J_0 v_2 . 
$$ 
We define
\begin{equation*}
{\mathcal V}_R=  \{ u \in H^{1}(\Omega_R)^3 \mid \nabla \cdot  u =0 \textrm{ in }\Omega_R,  \exists \ u_{\mathcal S}  \in {\mathcal R} \mbox{ such that } (u - u_{\mathcal S})\cdot n_{|\partial\Omega}=0 \text{ and } u_{|\partial B_R}= 0  \}
\end{equation*}
with scalar product 
\begin{equation*}
(u,v) \longmapsto 2 \int_{\Omega_R} {\mathbb D}(u):{\mathbb D}(v) \dx+ \alpha \int_{\partial \Omega} [u-u_{\mathcal S}]_{\tau}  \cdot [v-v_{\mathcal S}]_{\tau} \dS, 
\end{equation*}   
and induced norm
\begin{equation}
\label{nvR}
\|u\|_{\mc{V}_R}= \left( 2\| {\mathbb D}(u)\|^2_{2,\Omega_R} + \alpha \| [u-u_{\mathcal S}]_{\tau} \|^2_{2,\partial \Omega} \right)^{1/2}.
\end{equation}
{
Throughout the paper, we use the following identity several times
\begin{equation}\label{id}
(z\cdot \nabla) (b_u\times x)=b_u \times z,\quad z\in \mathbb{R}^3.
\end{equation}}

\begin{lemma}
\label{ledesfnl}
Let $\Omega$ be a $C^{1,1}$-domain and $ R > \delta({\mathcal B}_0) $.  Then  
\begin{eqnarray*}
\|\nabla  u\|_{2,\Omega_R} \leq C_1(\partial \Omega)  \|u\|_{\mc{V}_R},  \quad
\|u\|_{6,\Omega_R}\leq C_2(\partial \Omega,\alpha) \|u\|_{\mc{V}_R}, \quad
|a_u| + |b_u| \leq C_3(\partial \Omega,\alpha)  \|u\|_{\mc{V}_R} , 
\end{eqnarray*}
for all $u \in {\mathcal V}_R,$  where $u_{\mathcal S} (x) = a_u + b_u \times x$, and the constants $C_i$ are independent of $R$. Moreover
\[
\|u\|_{2,\Omega_R} \leq  C_4(\partial \Omega,R,\alpha) \|u\|_{\mc{V}_R}.
\]
\end{lemma}
\begin{proof} 
Consider $v:= u -  u_{\mathcal S}$. Since $\nabla \cdot v = 0$, we have 
\begin{equation}\label{lem3-1}
2 \|{\mathbb D}(v)\|^2_{2,\Omega_R} = \| \nabla v \|^2_{2,\Omega_R} + \int_{\partial \Omega_R}  n \cdot [(v \cdot \nabla) v] \dS.
\end{equation}
As $u\in \mc{V}_R$, we have $u = 0$ on $\partial B_R$ and the relation \eqref{lem3-1} becomes
\begin{equation}\label{lem3-2}
 2 \|{\mathbb D}(u)\|^2_{2,\Omega_R}=2 \|{\mathbb D}(v)\|^2_{2,\Omega_R} =  \| \nabla v \|^2_{2,\Omega_R} + \int_{\partial \Omega}  n \cdot [(v \cdot \nabla) v] \dS - \int_{\partial B_R} n \cdot [ (u_{\mathcal S}   \cdot \nabla) v] \dS.
\end{equation}
Observe that $v\cdot n=0$ on $\partial\Omega$. Thus we have $v=[u -u_{\mathcal S}]_\tau$ on $\partial\Omega$. As $[u -u_{\mathcal S}]_\tau=   n \times ([u -u_{\mathcal S}] \times n)$, we obtain
$\nabla_{\tau}\left([u -u_{\mathcal S}]_\tau \cdot n\right)=0 \mbox{ on }\partial\Omega$,
where $\nabla_{\tau}$ is the surface gradient on $\partial\Omega$.
As in  \cite[pg. 618, equation (42)]{defv}, we have  \begin{align}\label{lem3-3}
n \cdot [(v \cdot \nabla) v]=-[u -u_{\mathcal S} ]_\tau \cdot (\nabla_{\tau} n )  [u - u_{\mathcal S}]_\tau\mbox{ on }\partial\Omega.
\end{align}
As $v=u-u_{\mc{S}}$, we have 
\begin{equation}\label{lem3-3a}
- \int_{\partial B_R} n \cdot [ (u_{\mathcal S}   \cdot \nabla) v ] \dS = \int_{\partial B_R} n \cdot [ (u_{\mathcal S}   \cdot \nabla) u_{\mathcal S} ] \dS - \int_{\partial B_R} n \cdot [ (u_{\mathcal S}   \cdot \nabla) u ] \dS.
\end{equation}
We use $u_{\mathcal S} (x) = a_u + b_u \times x$ ($a_u,b_u \in {\mathbb R}^3$) and identity \eqref{id} to obtain 
\begin{equation*}
(u_{\mathcal S}   \cdot \nabla) u_{\mathcal S} = b_u \times a_u + b_u \times (b_u \times x).
\end{equation*}
Due to $\nabla\cdot (b_u \times a_u)=0$, we have
\begin{multline}\label{lem3-4}
 \int_{\partial B_R} n \cdot [ (u_{\mathcal S}   \cdot \nabla) u_{\mathcal S} ] \dS = \int_{\partial B_R} n \cdot [b_u \times (b_u \times x)] \dS = \int_{\partial B_R} n \cdot [(x\cdot b_u)b_u-|b_u|^2x] \dS \\ =\int_{ B_R} \nabla (x\cdot b_u)\cdot b_u \dx - |b_u|^2\int_{ B_R} \nabla \cdot x \dx
 = -2 |b_u|^2 |B_R|.
\end{multline}
We have the following relation between the surface divergence at $\partial B_R$ and the divergence in Cartesian coordinates {(see, for example,  \cite{Xu})
\[ \nabla_{\tau} \cdot  u  = \nabla \cdot u - [( n \cdot \nabla) u ] \cdot n \text{ on } \partial B_R. \]
Since $u = 0$ on $\partial B_R$ and $\nabla \cdot u = 0$ in $\Omega$, we get
$   [( n \cdot \nabla) u ] \cdot n = 0 \text{ on } \partial B_R$.
Therefore,
\begin{equation}\label{lem3-5}
\int_{\partial B_R} n \cdot [ (u_{\mathcal S}   \cdot  \nabla) u ] \dS = \int_{\partial B_R} n \cdot ( [u_{\mathcal S}]_\tau   \cdot \nabla_{\tau}  u) \dS  + \int_{\partial B_R} (u_{\mathcal S} \cdot n) 
 [( n \cdot \nabla) u ] \cdot n \dS = 0.
\end{equation}}
Using the relations \eqref{lem3-4} and  \eqref{lem3-5}, the identity \eqref{lem3-3a} becomes
\begin{equation}\label{lem3-6}
- \int_{\partial B_R} n \cdot [ (u_{\mathcal S}   \cdot \nabla) v] \dS = -2 |b_u|^2 |B_R|.
\end{equation}
Combining \eqref{lem3-3} and \eqref{lem3-6}, the relation \eqref{lem3-2} yields
\begin{equation}\label{lem3-7}
   2 \|{\mathbb D}(u)\|^2_{2,\Omega_R} + \int_{\partial \Omega}   [u -u_{\mathcal S} ]_\tau \cdot  (\nabla_{\tau}  n )[u - u_{\mathcal S}]_\tau \dS=  \| \nabla v \|^2_{2,\Omega_R} -2 |b_u|^2 |B_R|.
\end{equation}
As $v=u-u_{\mc{S}}$, we can write 
\begin{equation}\label{g3}
\| \nabla v \|^2_{2,\Omega_R}=  \| \nabla u \|^2_{2,\Omega_R} - \| \nabla u_{\mc{S}} \|^2_{2,\Omega_R}- 2 \int_{\Omega_R} \nabla u_{\mc{S}}:\nabla (u-u_{\mc{S}}) \dx
\end{equation}
By integration by parts and using the facts that $(u-u_{\mc{S}})\cdot n=0$ on $\partial\Omega$, $u=0$ on $\partial B_R$, we can compute
\begin{multline}\label{g2}
\int_{\Omega_R} \nabla u_{\mc{S}}:\nabla (u-u_{\mc{S}}) \dx = \int_{\partial\Omega_R} [(n\cdot\nabla)u_{\mc{S}}]\cdot (u-u_{\mc{S}}) \dS \\= \int_{\partial\Omega} [(n\cdot\nabla)u_{\mc{S}}]\cdot [u - u_{\mathcal S}]_\tau \dS - \int_{\partial B_R} [(n\cdot\nabla)u_{\mc{S}}]\cdot u_{\mc{S}} \dS.
\end{multline}
Using the identity \eqref{id}, we have
\begin{multline*}
[(n\cdot\nabla)u_{\mc{S}}]\cdot u_{\mc{S}} = (b_u \times n)\cdot (a_u + b_u\times x) = n\cdot (a_u \times b_u) + n\cdot [(b_u\times x)\times b_u]\\ = n\cdot (a_u \times b_u) + n\cdot [x|b_u|^2-(x\cdot b_u)b_u].
\end{multline*}
Thus,
\begin{equation}\label{g1}
\int_{\partial B_R} [(n\cdot\nabla)u_{\mc{S}}]\cdot u_{\mc{S}} \dS = \int_{B_R} \nabla\cdot \left((a_u \times b_u)+[x|b_u|^2-(x\cdot b_u)b_u]\right) = 2|b_u|^2 |B_R|.
\end{equation}
Using the relation \eqref{g1} and identity \eqref{id}, the relation \eqref{g2} becomes
\begin{equation}\label{g4}
\int_{\Omega_R} \nabla u_{\mc{S}}:\nabla (u-u_{\mc{S}}) \dx = \int_{\partial\Omega} (b_u\times n)\cdot [u - u_{\mathcal S}]_\tau \dS -2|b_u|^2 |B_R|.
\end{equation}
Since 
\begin{equation*}
\nabla u_{\mc{S}} = \begin{bmatrix}
0 & -(b_u)_3 & (b_u)_2\\
(b_u)_3 & 0 & -(b_u)_1 \\
-(b_u)_2 & (b_u)_1 & 0
\end{bmatrix},
\end{equation*}
we have from \eqref{g3} and \eqref{g4} that
\begin{equation*}
\| \nabla v \|^2_{2,\Omega_R} = \| \nabla u \|^2_{2,\Omega_R} - 2 |b_u|^2 |\Omega_R| + 4|b_u|^2 |B_R|  - 2  \int_{\partial\Omega} (b_u\times n)\cdot [u - u_{\mathcal S}]_\tau \dS.
\end{equation*}
Combining the above relation and \eqref{lem3-7}, we have
\begin{multline}\label{lem3-8}
\| \nabla u \|^2_{2,\Omega_R}-2 |b_u|^2 |B_R| - 2 |b_u|^2 |\Omega_R| + 4|b_u|^2 |B_R|  \\= 2  \int_{\partial\Omega} (b_u\times n)\cdot [u - u_{\mathcal S}]_\tau \dS+ \int_{\partial \Omega}   [u -u_{\mathcal S} ]_\tau \cdot  (\nabla_{\tau}  n )  [u - u_{\mathcal S}]_\tau \dS+  2 \|{\mathbb D}(u)\|^2_{2,\Omega} .
\end{multline}
We can estimate
{
\begin{equation}\label{lem3-9}
 \int_{\partial \Omega}   [u -u_{\mathcal S} ]_\tau \cdot (\nabla_{\tau}  n )[u - u_{\mathcal S}]_\tau \dS \leq C(\partial \Omega,\alpha) \alpha \| [u-u_{\mathcal S}]_{\tau} \|^2_{2,\partial \Omega},
\end{equation}
\begin{equation}
\label{lem3-9b}
 2  \int_{\partial\Omega} (b_u\times n)\cdot [u - u_{\mathcal S}]_\tau \dS \leq |b_u|^2 |{\mathcal B}_0| + \frac{|\partial \Omega|}{|{\mathcal B}_0|} \| [u-u_{\mathcal S}]_{\tau} \|^2_{2,\partial \Omega}.
\end{equation}
}
 Therefore, using the fact {$|B_R| - |\Omega_R| = |\mc{B}_0|$} and the relations \eqref{lem3-8}--\eqref{lem3-9b}, we obtain
 {
 \begin{equation*}
\| \nabla u \|^2_{2,\Omega_R} + |b_u|^2 |\mc{B}_0| \leq C(\partial \Omega,\alpha) \alpha \| [u-u_{\mathcal S}]_{\tau} \|^2_{2,\partial \Omega} +  2 \|{\mathbb D}(u)\|^2_{2,\Omega} .
\end{equation*}
It gives
\[
\begin{aligned}
 |b_u|^2 
 &  \leq \frac{C(\partial \Omega,\alpha)}{ |{\mathcal B}_0| }  \left[\alpha \| [u-u_{\mathcal S}]_{\tau} \|^2_{2,\partial \Omega}  +  2 \|{\mathbb D}(u)\|^2_{2,\Omega} \right] .
 \end{aligned}
\]
}
Since $u=0$ on $\partial B_R$, we have
\[
\|u\|^2_{2,\Omega_R} \leq C(\Omega_R) \| \nabla u \|^2_{2,\Omega_R}  \leq C(\Omega_R) \left[C(\partial \Omega,\alpha) \alpha \| [u-u_{\mathcal S}]_{\tau} \|^2_{2,\partial \Omega} +  2 \|{\mathbb D}(u)\|^2_{2,\Omega}\right].
\]

It remains to bound $|a|$ in terms of $\left( 2\| {\mathbb D}(u)\|^2_{2,\Omega_R} + \alpha \| [u-u_{\mathcal S}]_{\tau} \|^2_{2,\partial \Omega} \right)^{1/2}$. We have
\[
 u \cdot n  = a \cdot n + (b \times x )\cdot n, \text{ on } \partial \Omega.
\]
We fix $R_0>0$ such that $\delta (\mathcal{B}_0) < R_0 < R$. We have
\begin{align*}
 \| a \cdot n \|_{2,\partial \Omega}   \leq  & \, \| u \cdot n \|_{1/2,2,\partial \Omega}  +  \| (b \times x )\cdot n \|_{1/2,2,\partial \Omega} \\
 \leq & \, C(\Omega_{R_0}) (\| u \|_{2,\Omega_{R_0}} +  \|\nabla  u \|_{2,\Omega_{R_0}})  +  C(\partial \Omega) |b| \\ 
 \leq & \, C(\Omega_{R_0}) (\| u \|_{6,\Omega_{R_0}} +  \|\nabla  u \|_{2,\Omega_{R_0}}) + C(\partial \Omega) |b|  \\ 
 \leq & \, C(\Omega_{R_0})   \|\nabla  u \|_{2,\Omega_R} +  C(\partial \Omega) |b| \leq C(\partial \Omega,R_0)  \left( 2\| {\mathbb D}(u)\|^2_{2,\Omega_R} + \alpha \| [u-u_{\mathcal S}]_{\tau} \|^2_{2,\partial \Omega} \right)^{1/2}. 
 \end{align*}
We would like to prove 
\[
|a| \leq C \| a \cdot n \|_{2,\partial \Omega}, \,\ \forall \ a \in {\mathbb R}^3.
\]
It is enough to show that  $ {\mathbb R}^3 \ni a \mapsto \| a \cdot n \|_{2,\partial \Omega}$, is a norm in $ {\mathbb R}^{3}$. Indeed we have to show
\[ \| a \cdot n \|_{2,\partial \Omega} =0 \implies a =0,
\]
which implies we need to show
 \[ a \cdot n = 0 \text{ on } \partial \Omega \implies a =0 .
\]
Let us define $\zeta(x):= a(a\cdot x)$, $x\in \R^3$. Observe that 
\begin{equation*}
\int_{\mc{B}_0} \nabla\cdot \zeta \dx = \int_{\partial\Omega} (a\cdot n)(a\cdot x) \dS
\end{equation*}
As $a \cdot n = 0$ on $\partial \Omega$ and $\nabla\cdot \zeta = |a|^2$, we conclude that $a=0$.
 \end{proof}
 \begin{remark}
 It is possible to extend $u \in {\mathcal V}_R$ by zero to $B^R$, but unlike the Dirichlet case, $u$ cannot be extend by $u_{\mathcal S}$ to ${\mathcal B}_0$ as a $H^1({\mathbb R}^3)^3$ function, because we only know that $u \cdot n =  u_{\mathcal S} \cdot n$ on  $\partial \Omega$. {Indeed, if we define} 
  \begin{equation}
v(x) = \left\{  \begin{array}{l}  a+b\times x = u_{\mathcal S}(x), \, x \in {\mathbb R}^3 \setminus \Omega = \mc{B}_0, \medskip \\  u(x) , \, x \in  \Omega, \end{array}\right.
\label{defv}
\end{equation}
we have
\[
\nabla v = \chi_{\Omega_R} \nabla u + \chi_{{\mathcal B}_0} \nabla u_{\mathcal S}  + n \otimes  (u_{\mathcal S}  - u)  \delta_{\partial \Omega}, 
\]
with $(u_{\mathcal S} - u)\cdot n = (a+b\times x - u) \cdot n = 0$ at $\partial \Omega$, and therefore
$$
\nabla \cdot v = \chi_{\Omega_R} \nabla \cdot u + \chi_{{\mathcal B}_0} \nabla \cdot  u_{\mathcal S}  + n \cdot (u_{\mathcal S}  - u)  \delta_{\partial \Omega} = 0. 
$$
However, for the gradient we only know
\[
\nabla v= \chi_\Omega \nabla u + \chi_{{\mathcal B}_0} \nabla u_{\mathcal S}  + n \otimes  [u_{\mathcal S} - u]_\tau  \delta_{\partial \Omega}. 
\]
\end{remark}
\subsection{Construction of basis}

Next, we examine an auxiliary steady problem. Recall the norm \eqref{nvR} defined in the space ${\mathcal V}_R$, and that for each $f \in {\mathcal H}_R$ there exists an associated element $f_{\mathcal S} \in {\mathcal R}$ such that $f\cdot n = f_{\mathcal S}\cdot n$ on $\partial \Omega$, $f_{\mathcal S}$ being determined by $a_f,b_f \in {\mathbb R}^3$ such that $f_{\mathcal S}(x)=a_f + b_f \times x$.

\begin{lemma}
Let $\partial \Omega$ be of class $C^{1,1}$ and $f \in {\mathcal H}_R$. Then the problem
\begin{equation}
\left\{
\begin{aligned}
- \nabla \cdot {\mathbb T}(u,p)  = f, \quad
\nabla \cdot  u  = 0 \quad  \mbox{ in }\Omega_R, \\ 
u \cdot n  = V \cdot n \quad \mbox{ on }\partial\Omega,  \\ 
2 [{\mathbb D}(u) n] \times n + \alpha u \times n 
 = \alpha V \times n \quad \mbox{ on }\partial\Omega, \\
u  = 0 \quad \mbox{ on } \partial B_R, \\ 
\displaystyle m a_f  =  \int\limits_{\partial\Omega} {\mathbb T}(u,p) n \dS, \\ 
\displaystyle  J_0 b_f  =  \int\limits_{\partial\Omega}  x\times  {\mathbb T}(u,p) n \dS,
\end{aligned}
\label{systoke}
\right.
\end{equation}
has a unique solution $(u,V,p) \in {\mathcal{V}}_R \cap H^2(\Omega_R)^3 \times {\mathcal R} \times H^1(\Omega_R)$, with $V(x)=u_{\mathcal S}(x)=a_u + b_u \times x$ and $\displaystyle \int_{\Omega_{R_0}}p =0,$ for fixed $R_0<R$. Moreover, the following estimates hold
\begin{eqnarray*}
\displaystyle   \| u \|_{{\mathcal V}_R}   &\leq & C(\partial \Omega,\alpha,R) \left( \|f\|^2_{2,\Omega_R} + m |a_f|^2 + b_f \cdot J_0 b_f \right)^{1/2}  \\
|a_u| + |b_u|&\leq & C(\partial \Omega)  \| u \|_{{\mathcal V}_R}  \\
\|p\|_{2,\Omega_{R_0}}&\leq & C(\partial \Omega,R_0,\alpha)(\ \|f\|_{2,\Omega_R} +  \| u \|_{{\mathcal V}_R}  ) \\
\|D^2 u\|_{2,\Omega_R} + \|\nabla  p \|_{2,\Omega_R} & \leq & C(\partial \Omega,R_0,\alpha)(\|f\|_{2,\Omega_R} + \| u \|_{{\mathcal V}_R} )
\end{eqnarray*} 
\end{lemma}
\begin{proof} Let us consider the following problem: given $f \in {\mathcal H}_R$,  find  $u \in {\mathcal{V}}_R$ such that 
\begin{equation}
\label{eqpf}
 2 \int_{\Omega_R} {\mathbb D}(u): {\mathbb D}(\varphi ) \dx + \alpha \int_{\partial \Omega} [u-u_{\mathcal S}]_{\tau}  \cdot [\varphi - \varphi_{\mathcal S}]_{\tau}\dS = \ \int_{\Omega_R} f \cdot \varphi \dx + m a_f \cdot a_\varphi + b_f \cdot J_0 b_\varphi, \quad \forall \varphi \in {\mathcal{V}}_{R}.
\end{equation}
Using the estimates of Lemma \ref{ledesfnl}, and applying Lax-Milgram Theorem, we conclude that the problem \eqref{eqpf} has  unique solution $u \in \mathcal{V}_{R}$ and satisfies
\begin{equation*}
  \| u \|_{{\mathcal V}_R} \leq C(\partial \Omega, \alpha, R) \left(  \|f\|^2_{2,\Omega_R} + m |a_f|^2 + b_f \cdot J_0 b_f \right)^{1/2}.
\end{equation*}
Since $u \in \mathcal{V}_{R}$, it is clear that $u \in H^{1}({\Omega}_R)$ and that it satisfies \eqref{systoke}$_{2}$, \eqref{systoke}$_{3}$ {with $V  = u_{\mathcal S}$}, and \eqref{systoke}$_{4}$.  Since $\mathcal{V}_{R}$ contains the space of all solenoidal vector functions of ${\mathcal D}(\Omega_R)^3$, we have 
$$
 2 \int_{\Omega_R} {\mathbb D}(u):{\mathbb D}(\varphi) \dx  = \ \int_{\Omega_R} f \cdot \varphi \dx , \quad  \forall\ \varphi \in {\mathcal D}(\Omega_R)^3 \text{ with } \nabla \cdot  \varphi =0.
$$
Therefore there exists $p \in L^{2}(\Omega_R)$, which is unique if $\displaystyle \int_{\Omega_{R_0}}p \dx = 0,$ such that 
\[
  \int_{\Omega_R} \left( 2 {\mathbb D}(u):{\mathbb D}(\varphi)  -  p \nabla \cdot \varphi \right) \dx =\ \int_{\Omega_R} f \cdot \varphi \dx, \quad  \forall \varphi \in  {\mathcal D}(\Omega_R)^3,
\]
that is,
\begin{equation}
- \nabla \cdot  {\mathbb T}(u,p) = \ f \text{ in } {\mathcal D}'(\Omega_R)^3.
\label{rrrr}
\end{equation}

We have $\nabla \cdot  {\mathbb T}(u,p) \in L^2(\Omega_R)^3$ because $f \in L^2(\Omega_R)^3$ and $2 {\mathbb D}(u) - p {\mathbb I} \in L^2(\Omega_R)^{3\times 3}$. Let $\varphi \in {\mathcal V}_{R}$ be such that $\varphi \cdot n = 0$ at $
\partial \Omega$. Taking the scalar product of (\ref{rrrr}) with this $\varphi $  and integrating by parts, we find 
\begin{equation*}
 - 2 \int_{\partial \Omega}[\varphi]_{\tau} \cdot  [{\mathbb D}(u) n]_\tau  \dS + 2 \int_{\Omega_R} {\mathbb D}(u): {\mathbb D}(\varphi ) \dx = \ \int_{\Omega_R} f\cdot \varphi \dx.
\end{equation*}
Taking the same $\varphi$ in \eqref{eqpf} and comparing the two equations, we conclude:
\begin{equation}
\int_{\partial \Omega} \left(  2 [{\mathbb D}(u) n]_\tau + \alpha  [u-V]_{\tau} \right) \cdot [\varphi]_\tau \ \dS = 0, \quad \forall \varphi \in {\mathcal V}_{R}
\text{ such that } \varphi  \cdot n = 0,
\label{eqvh}
\end{equation}
The next argument was used in \cite{AcAmConGh} and allows us to extend \eqref{eqvh} to any $\psi \in H^{1/2}(\partial \Omega)^3$. Given any $\psi \in H^{1/2}(\partial \Omega)^3$, there exists $\Psi \in H^1(\Omega)^3$ such that $\nabla \cdot \Psi=0$ in $\Omega$, $\Psi = 0$ on $\partial B_R$ and $\Psi = [\varphi]_\tau$ on $\partial \Omega$, that is $\Psi \in {\mathcal V}_{R}$. Then
$$
\int_{\partial \Omega} \left(  2 [{\mathbb D}(u) n]_\tau + \alpha  [u-V]_{\tau} \right) \cdot \psi \dS = \int_{\partial \Omega} \left( 2 [{\mathbb D}(u) n]_\tau + \alpha  [u-V]_{\tau} \right) \cdot [\psi]_\tau \dS = 0, \quad  \forall \psi \in H^{1/2}(\partial \Omega),
$$
meaning that 
$$
2 [{\mathbb D}(u) n]_\tau + \alpha  [u-V]_{\tau}  = 0 \text{ on } \partial \Omega.
$$

Now, for $i \in \{1,2,3\}$, let $\varphi \in {\mathcal V}_{R}$ be such that $\varphi \cdot n = \textsf{e}_i \cdot n$. Testing (\ref{rrrr}) with this $\varphi $  and integrating by parts, we find 
\begin{equation}
\textsf{e}_i \cdot \int_{\partial\Omega} {\mathbb T}(u,p ) n \dS - 2 \int_{\partial\Omega}[\varphi - \textsf{e}_i]_{\tau} \cdot  [{\mathbb D}(u) n]_\tau  \dS + 2 \int_{\Omega_R} {\mathbb D}(u): {\mathbb D}(\varphi ) \dx = \ \int_{\Omega_R} f\cdot \varphi \dx
\label{eqvarast}
\end{equation}
where we used
\[ 
\begin{aligned}
- \int_{\partial\Omega} \varphi  \cdot {\mathbb T}(u,p ) n \dS & =  - \textsf{e}_i \cdot \int_{\partial\Omega} {\mathbb T}(u,p ) n \dS - \int_{\partial\Omega}(\varphi - \textsf{e}_i ) \cdot  {\mathbb T}(u,p ) n  \dS \\ 
& = - \textsf{e}_i \cdot \int_{\partial\Omega} {\mathbb T}(u,p ) n \dS - 2 \int_{\partial\Omega}[\varphi - \textsf{e}_i]_{\tau} \cdot  [{\mathbb D}(u) n]_\tau  \dS.
\end{aligned}
\]
Taking this $\varphi$ in \eqref{eqpf} and taking into account that 
\[
 2 \int_{\partial\Omega}[\varphi - \textsf{e}_i]_{\tau} \cdot  [{\mathbb D}(u) n]_\tau  \dS - \alpha \int_{\partial \Omega} [u-V]_{\tau}  \cdot [\varphi - \textsf{e}_i]_{\tau} \dS = 0,
\]
yields
$$
2 \int_{\Omega_R} {\mathbb D}(u): {\mathbb D}(\varphi) \dx - 2 \int_{\partial\Omega}[\varphi - \textsf{e}_i]_{\tau} \cdot  [{\mathbb D}(u) n]_\tau  \dS = \ \int_{\Omega_R} f \cdot \varphi \dx + m a_f \cdot \textsf{e}_i 
$$
and comparing the latter equation with \eqref{eqvarast}, we conclude that 
$$
m a_f =  \int_{\partial\Omega} {\mathbb T}(u,p) n \dS.
$$
Analogously, taking $\varphi \in {\mathcal V}_{R}$ such that $\varphi_{\mathcal S} = \textsf{e}_i \times x,$ we obtain (\ref{systoke})$_{6}.$ 

The estimates for the second derivatives of $u$ and for the gradient of $p$ are obtained from regularity results for the Stokes problem with Navier boundary conditions, see  \cite[Theorem 4.5]{AcAmConGh}, along with the following estimates (see Lemma \ref{ledesfnl})
$$
\| \nabla  u \|_{2,\Omega_R} + |a_u| + |b_u| \leq C(\partial \Omega,R_0,\alpha) 
  \| u \|_{{\mathcal V}_R} . 
$$
The Gagliardo-Nirenberg inequality in bounded domains, along with the previously obtained estimates, leads to the remaining estimates.
\end{proof}


Using the previous lemma, we can now establish the existence of special orthonormal bases in ${\mathcal H}_{R}$ and ${\mathcal V}_R$.
\begin{lemma}
\label{tvp}
The problem  
\begin{equation}
\left\{
\begin{aligned}
- \nabla \cdot  {\mathbb T}(w,p)=\lambda w,  \quad 
\nabla \cdot w =0 
\quad \mbox{ in }\Omega_R, \\ 
w \cdot n = w_{\mathcal S} \cdot n \quad \mbox{ on } \partial \Omega, \\ 
2 [{\mathbb D}(w) n] \times n + \alpha w \times n 
= \alpha w_{\mathcal S} \times n \quad \mbox{ on }\partial\Omega, \\ 
w  = 0 \quad \mbox{ on } \partial B_R,  \\ 
 \lambda m  a_w = \int\limits_{\partial\Omega}  {\mathbb T} (w,p) n \dS, \\ 
\lambda J_0  b_w = \int\limits_{\partial\Omega}  x \times {\mathbb T}(w,p) n \dS,
\end{aligned}
\right.
\label{pvp}
\end{equation}
admits a sequence $\{\lambda
_{i}\}_{i \in {\mathbb N}}$ of eigenvalues 
\[
0 < \lambda_1 \leq \lambda_2 \leq \lambda_3 \leq ...
\]
clustering in infinity and the corresponding eigenfunctions
$\{w_{i}\}_{i \in {\mathbb N}} \subset {\mathcal V}_R \cap H^{2}(\Omega_R)^3$ form an orthonormal basis of ${\mathcal H}_R$. Moreover, $\left\{ \Phi_i \right\}_{i \in {\mathbb N}} := \left\{\frac{w_i}{\sqrt{\lambda_i}}\right\}_{i \in {\mathbb N}}$ is an orthonormal basis of ${\mathcal V}_R$.  
\end{lemma}
\begin{proof} The operator $\Lambda :f \mapsto u$ associated with the problem \eqref{systoke} is linearly continuous from ${\mathcal H}_R$ to ${\mathcal V}_R$ and since ${\mathcal V}_R$ is compactly embedded in ${\mathcal H}_R,$ $\Lambda$ is a compact operator in ${\mathcal H}_R$. This operator is also self-adjoint and positive definite and, therefore, there is an orthonormal basis $\{w_{i}\}_{i \in {\mathbb N}}$ of ${\mathcal H}_R$ of eigenfunctions associated with the eigenvalues of $\Lambda$. These eigenvalues are positive and form a decreasing sequence $\{ \mu_{i}\}_{i \in {\mathbb N}}$  such that $\mu_i \rightarrow 0 ,$ as $i \rightarrow \infty.$ Therefore, we have
\begin{equation}
\alpha \int_{\partial \Omega} [w_{i} - {(w_{i}})_{\mathcal S}]_{\tau}  \cdot [\varphi - {\varphi}_{\mathcal S}]_{\tau} \dS + 2 \int_{\Omega} {\mathbb D}(w_{i}):{\mathbb D}(\varphi ) \dx =  \lambda_i \bigg( \int_{\Omega _R} w_{Ri} \cdot  \varphi + m {a_{w_{i}}} \cdot a_\varphi + {b_{w_{i}}} \cdot J_0 b_\varphi  \bigg) , 
\label{phpv}
\end{equation}
for all $\varphi \in {\mathcal{V}}_{R}$ where $\lambda_i = 1/ \mu_i.$ Taking $\varphi = w_j$ in \eqref{phpv} yields that $\left\{\frac{w_i}{\sqrt{\lambda_i}}\right\}_{i \in {\mathbb N}}$ is orthonormal in ${\mathcal V}_R$. 

Let $\Phi_{i} := \frac{w_i}{\sqrt{\lambda_i}}$ and suppose that $u \in {\mathcal V}_R$ satisfies $\alpha \int\limits_{\partial \Omega} [\Phi_{i} - {(\Phi_{i}})_{\mathcal S}]_{\tau}  \cdot [u - u_{\mathcal S}]_{\tau}\ dS + 2 \int_{\Omega} {\mathbb D}(\Phi_{i}):{\mathbb D}(u) dx = 0$, for all $i \in \mathbb N$. From \eqref{phpv}  it follows that $\int_{\Omega _R} \Phi_{Ri} \cdot  u dx  + m {a_{\Phi_{i}}} \cdot a_u + {b_{\Phi_{i}}} \cdot J_0 b_u = 0$, for all $i \in \mathbb N$, and since $\{w_{i}\}_{i \in {\mathbb N}}$ is a basis of ${\mathcal H}_R$, we deduce $u \equiv 0$. Therefore, $\left\{\Phi_i\right\}_{i \in {\mathbb N}}$ forms an orthonormal basis of ${\mathcal V}_R$.
\end{proof}

Next, we give a density result for the space ${\mathcal V} $. The functions of ${\mathcal V}_R$ are extended to zero in $\Omega^{R}$.
\begin{lemma}
{$\displaystyle \cup_{R > \delta({\mathcal B}_0)}{\mathcal V}_R$ is dense in ${\mathcal V} $ and $\displaystyle \cup_{R > \delta({\mathcal B}_0)}\{ \Phi_{R,i}, i \in {\mathbb N} \}$ is a basis in ${\mathcal V} .$ Moreover, ${\mathcal V} \subset H^{1,0}_\rho(\Omega)$.}
\label{denseh} 
\end{lemma}
\begin{proof} {Given $u \in  {\mathcal V}$, there exists an increasing and unbounded sequence $\{ R_k \}_{k \in {\mathbb N}}$ and a sequence $\{\varphi_k\}_{k \in {\mathbb N}} \subset C^\infty_0(\Omega)$ of divergence free functions converging to $u$ in ${\mathcal V}$ such that, for each $k \in \mathbb N$, $\textrm{supp } \varphi_k \subset \overline{\Omega_{R_k}},$ $\varphi_k=0$ on $\partial B_{R_k}$, and $\varphi_k \cdot n =(\varphi_k)_{\mathcal S} \cdot n$ for some $(\varphi_k)_{\mathcal S} \in {\mathcal R}$. It is clear that $\varphi_k \in {\mathcal V}_{R_k}$ and therefore we found a sequence of functions in $\displaystyle \cup_{R > \delta({\mathcal B}_0)}{\mathcal V}_R$ converging to $u$ in  ${\mathcal V}$. The fact that $\displaystyle \cup_{R > \delta({\mathcal B}_0)}\{ \Phi_{R,i}, i \in {\mathbb N} \}$ is a basis in ${\mathcal V} $ follows from Lemma \ref{tvp}. Finally, ${\mathcal V} \subset H^{1,0}_\rho(\Omega)$ is a consequence of  \cite[Theorem II.6.1]{G}.} \end{proof}

\medskip 

{From now on, we will use the following simplified notations:
\[
u \cdot \nabla v  = (u \cdot \nabla) v,  \qquad  u \cdot \nabla v \cdot w = [(u \cdot \nabla) v ] \cdot w,  \qquad  u \cdot (\nabla v)^\top \cdot w = [(w \cdot \nabla) v ] \cdot u.
\]}
\subsection{Extension of the normal component of the boundary values}
Suppose there is an extension $\widetilde{[v_*]_n} \in H^1( \Omega)^3$ of $[v_*]_n:=(v_* \cdot n)n$ with the property that $\nabla \cdot \widetilde{[v_*]_n}   =  0$ in $\Omega$. Then the velocity can be decomposed as  $v = u + \widetilde{[v_*]_n}$ such that $(u-V) \cdot n = 0$ on $\partial\Omega$. Let us set the following notations 
\begin{equation*}\label{notation1}
 \Phi=\int_{\partial \Omega}  v_* \cdot n\ dS \not=0,\quad {\mathbb W}(u)  := \frac 12 (\nabla u - (\nabla u)^\top),
\end{equation*} 
and 
{
\begin{equation*}\label{notation2}
\sigma(x)= - \frac{x - x_0}{4 \pi |x-x_0|^3}, \qquad 
x_0 =
\begin{cases}
0, & \text{when } 0 \notin \Omega,\\[4pt]
\text{any point in }\mathrm{Int}(\mc B_0), & \text{when } 0 \in \Omega.
\end{cases}
\end{equation*}}
Precisely, we have the following extension result.
{
\begin{lemma}
\label{lemaext}
Let $v_{*}\cdot n \in H^{\frac{1}{2}}(\partial \Omega)$. There exists an extension $\widetilde{[v_*]_n}  \in H^1( \Omega)^3$ of $(v_{*}\cdot n )n$ satisfying 
\begin{align*}
 \begin{cases}
  \nabla \cdot \widetilde{[v_*]_n}   =  0 \ &\text{ in }  \Omega, \\
 \widetilde{[v_*]_n}  =  (v_* \cdot n) n\  &\text{ on }  \partial \Omega
    \end{cases}
    \end{align*}
    and the estimate
    \begin{equation}
\| \widetilde{[v_*]_n}  \|_{1,2,\Omega} \leq C(\Omega)\| v_{*} \cdot n \|_{1/2,2,\partial \Omega}.
\label{esttvn}
\end{equation}
Moreover, 
\begin{multline}\label{est:extension}
 2 \,  \mathrm{Re}  \left|  \int_{\Omega_R} (u - V) \cdot {\mathbb W}(u) \cdot \widetilde{[v_*]_n} \dx  \right| \\
\leq 
\mathrm{Re}  \, C_1({\mathcal B}_0) \| (v_{*}\cdot n )n - \Phi \sigma|_{\partial\Omega} \|_{1/2,\partial\Omega}\| u \|_{{\mathcal V}_R}^2 + 
\mathrm{Re}   \, C_2({\mathcal B}_0) | \Phi  |   \| u \|_{{\mathcal V}_R}^2 , \, \forall\ u \in {\mathcal V}_R,
\end{multline}
where the constants $C_i({\mathcal B}_0)$ are independent of $R$, and $V=u_{\mathcal S} \in {\mathcal R}$.
\end{lemma}
    \begin{proof} Let us define $\beta_* \in H^{\frac{1}{2}}(\partial \Omega)^3$ by
$\beta_*:=  (v_{*} \cdot n) n  -  \Phi \sigma|_{\partial \Omega}$.  Recall the Laplace fundamental solution
\[
{\mathcal E} (x) := \frac{1}{4 \pi |x|}, \qquad -\Delta {\mathcal E}  = \delta_0 \text{ in } {\mathbb R}^3.
\]
We have
 $$\nabla \cdot \sigma(x) =  \Delta {\mathcal E} (x-x_0) = - \delta_{x_0}(x)  \implies \int_{\partial \Omega} \sigma \cdot n \,
\dS = 1$$ and therefore 
 \begin{equation*}\label{lem4.3-1}
 \int_{\partial \Omega} \beta_* \cdot n \dS = 0 \text{ and } \nabla \cdot \left(  \Phi \sigma \right) =0 \text{ in }\Omega.
 \end{equation*} 
 
 For fixed $R_0 > \delta(\mathcal{B}_0)$, consider $$
 \vartheta (x) = \begin{cases}  v_{\Omega_{R_0}}(x) , & \quad x \in \Omega_{R_0},  \\  0, & \quad x \in \overline{B^{R_0}}, \end{cases}
 $$
where $v_{\Omega_{R_0}} \in H^1(\Omega_{R_0})$ is a solenoidal extension of $\beta_*$ in $\Omega_{R_0}$ with $v_{\Omega_{R_0}}|_{\partial B_{R_0}}=0$. Then $\vartheta \in H^1(\Omega)$,   $\nabla \cdot \vartheta  = 0 $ in $\Omega$, and the following estimates are valid
\[
\| \vartheta \|_{1,2,\Omega} \leq C(\partial \Omega, R_0) \|  \beta_* \|_{1/2,2,\partial \Omega} \leq C(\partial \Omega, R_0) \|  v_* \|_{1/2,2,\partial \Omega}.
\]
The desired extension of $(v_{*} \cdot n) n $ is given by:
\[
 \widetilde{[v_*]_n}  = \vartheta  +  \Phi \sigma
\]
and it clearly satisfies estimate \eqref{esttvn} because $\sigma \in H^1(\Omega)^3$.

For $R > R_0$, we have for  $u \in {\mathcal V}_R$, $V=u_{\mathcal S} \in {\mathcal R}$,
\begin{align*}
\int_{\Omega_R}  (u  - V) \cdot {\mathbb W}(u) \cdot \widetilde{[v_*]_n} \dx 
 &=   \int_{\Omega_{R_0}}   (u - V) \cdot {\mathbb W}(u)  \cdot  \vartheta \dx +  \, \Phi   \int_{\Omega_R} (u - V)  \cdot {\mathbb W}(u)  \cdot \sigma \dx \\ &=: I_1 + I_2.
\end{align*}

Using H\"{o}lder inequality, the estimate for $\vartheta$ and Korn inequality, we obtain
\begin{equation}\label{est:I1}
\left| I_1 \right|  
 \leq C_1(\Omega_{R_0}) \| (v_{*}\cdot n )n - \Phi \sigma|_{\partial\Omega} \|_{1/2,\partial\Omega}\| u \|_{{\mathcal V}_R}^2. 
\end{equation}    

By direct calculations, one finds 
\begin{equation}
\sigma \in L^q(\Omega)^3, \, q > 3/2 \text{ and } \nabla \sigma \in L^q(\Omega)^{3 \times 3}, \, q>1,
\label{sumsig}
\end{equation}
and
\begin{equation}
(\omega \times (x-x_0))  \cdot \nabla  \sigma(x) -  \omega \times \sigma(x) = 0 .
\label{divsig}
\end{equation}

After integration by parts and using the relation \eqref{divsig}, the integral $I_2$ can be written as
\[
\begin{aligned} 
I_2 
 =& \,
 \frac 12  \int_{\Omega_R} (u - V)  \cdot \nabla u  \cdot \sigma \dx -  \frac 12  \int_{\Omega_R} \sigma \cdot  \nabla u \cdot (u - V)  \dx\\
 =&
  - \frac 12   \int_{\Omega_R} u \cdot \nabla  \sigma \cdot u \dx  + \frac 12   \int_{\Omega_R}   \sigma \cdot \nabla u \cdot u \dx \\
& + \frac 12   \int_{\Omega_R} \xi \cdot \nabla  \sigma \cdot u \dx + \frac 12   \int_{\Omega_R} [ (\omega \times x)  \cdot \nabla  \sigma -  \sigma \cdot \nabla (\omega \times x) ] \cdot u \dx \\
= & - \frac 12   \int_{\Omega_R} u \cdot \nabla  \sigma \cdot u \dx  + \frac 12   \int_{\Omega_R}   \sigma \cdot \nabla u \cdot u \dx + \frac 12   \int_{\Omega_R} (\xi + \omega \times x_0) \cdot \nabla  \sigma \cdot u \dx.
 \end{aligned}
\]
Using \eqref{sumsig} and Lemma \ref{ledesfnl}, we have  
\begin{equation}\label{est:I2}
 \left| I_2 \right|  \leq C_2(\Omega) | \Phi  |   \| u \|_{{\mathcal V}_R}^2.
\end{equation}
By combining, \eqref{est:I1}--\eqref{est:I2}, we obtain the desired estimate \eqref{est:extension}.
\end{proof}}
\begin{remark}
This extension of the boundary values $(v_* \cdot n)n$ does not require $\Phi=0$, as was the case in \cite[Theorem 5.1]{Galdi1999}. {However, using only the boundary condition $(u-V)\cdot n = 0$, we were not able to solve the more general problem: for a given $\gamma$, construct an extension $\widetilde{[v_*]_n} $ of $[v_*]_n$ satisfying: 
\[
2 \,  \mathrm{Re}  \left|  \int_{\Omega_R} (u - V) \cdot {\mathbb W}(u) \cdot \widetilde{[v_*]_n} \dx  \right| 
\leq 
\gamma + 
\mathrm{Re}   \, C({\mathcal B}_0) | \Phi  |   \| u \|_{{\mathcal V}_R}^2 , \, \forall\ u \in {\mathcal V}_R.
\] 
The difficulty lies in the fact that $[u-V]_\tau \not= 0$, which prevents the use of Hardy inequality as in the proof of \cite[Lemma IX.4.2, (IX.4.42)-(IX.4.43)]{G}.}
\end{remark}
\subsection{Approximating solutions and existence of a weak solution}
In this section, we will present the proof of Theorem \ref{Mainresult3}.
We consider a sequence of bounded domains $\Omega_R:= \Omega \cap B_R$ that invade $\Omega$. In each domain $\Omega_R$, with $R$ fixed, we utilize the basis 
$\{ \Phi_{R,i}, i\in {\mathbb N}\}$ of ${\mathcal V}_R$, as introduced in Lemma \ref{tvp}. 
\begin{proof}[Proof of Theorem \ref{Mainresult3}]
We can use the decomposition of velocity $v = u + \widetilde{[v_*]_n}$ such that $(u-V) \cdot n = 0$ on $\partial\Omega$ and rewrite the weak formulation \eqref{weakform} in the following way 
\begin{equation}
\label{re:weakform}
\begin{aligned}
 & 2\int_{\Omega_{R}} {\mathbb D}({u}): {\mathbb D}({\varphi})\dx + \alpha \int_{\partial \Omega} [u-V]_{\tau}  \cdot [\varphi-\varphi_{\mathcal S}]_{\tau}\dS \\
 =  & \,  \mathrm{Re} \int_{\Omega_{R}}(u-V) \cdot \nabla \varphi \cdot u \dx -  \mathrm{Re}  \int_{\Omega_{R}} ( \omega \times u ) \cdot \varphi \dx  + \mathrm{Re} \, m\xi\times\omega \cdot a_\varphi  \\
 &  + \mathrm{Re} \, J_0\omega\times\omega \cdot b_\varphi 
  +  \mathrm{Re} \int_{\Omega_{R}} \widetilde{[v_*]_n} \cdot \nabla \varphi \cdot u \dx + \mathrm{Re} \int_{\Omega_{R}}(u-V) \cdot \nabla \varphi \cdot \widetilde{[v_*]_n} \dx \\
 & - \mathrm{Re} \int_{\Omega_{R}} ( \omega \times \widetilde{[v_*]_n}) \cdot \varphi \dx - \mathrm{Re} \int_{\partial \Omega} (v_{*}\cdot n)[u]_{\tau}  \cdot [\varphi-\varphi_{\mathcal S}]_{\tau} \dS 
  + \mathrm{Re} \int_{\Omega_{R}} \widetilde{[v_*]_n} \cdot \nabla \varphi \cdot \widetilde{[v_*]_n} \dx \\ 
  & - 2\int_{\Omega_{R}} {\mathbb D}(\widetilde{[v_*]_n}): {\mathbb D}({\varphi})\dx  
  +  \alpha \int_{\partial \Omega} [v_*]_{\tau}  \cdot [u - V]_{\tau}\dS , \quad \varphi \in {\mathcal C}(\Omega).
 \end{aligned}
\end{equation}
 An approximating solution will be constructed in the bounded domain $\Omega_R$, $R > \delta ({\mathcal B}_0)$, using the Galerkin method. The sequence of approximating solutions $\{(u_{R}^{(k)} , \xi_{R}^{(k)}, \omega_{R}^{(k)})\}_{k \in {\mathbb N}}$ is defined by
\begin{equation*} 
u_{R}^{(k)}(x)= \sum _{i=1}^{k} c_{R,i}^{(k)} \Phi_{R,i} (x), \quad \xi_{R}^{(k)}= \sum _{i=1}^{k}  c_{R,i}^{(k)} a_{\Phi_{R,i}}, \quad \omega_{R}^{(k)} = \sum _{i=1}^{k}   c_{R,i}^{(k)} b_{\Phi_{R,i}}, \quad V_{R}^{(k)}(x) = \xi_{R}^{(k)} + \omega_{R}^{(k)} \times x =  (u_{R}^{(k)})_{\mathcal S}.
\end{equation*}
To simplify the notation, we omit the dependence on $R$ in what follows, that is, we write $\Phi_{i}:=\Phi_{R,i}$, $c_{i}^{(k)}:=c_{R,i}^{(k)}$, and $u^{(k)} := 
u_{R}^{(k)}$, $\xi^{(k)} :=\xi_{R}^{(k)}$, $\omega_{k}:=\omega_{R}^{(k)}$. Motivated by \eqref{re:weakform}, the coefficients $c_{k}^{(i)}$, $i=1,...,k$, are determined by solving the following nonlinear system: for $j=1,...,k$, 
\begin{equation}
\label{sistGR}
\begin{aligned}
  & 2\int_{\Omega} {\mathbb D}({u^{(k)}}): {\mathbb D}({\Phi_{j}})\dx + \alpha \int_{\partial \Omega} [u^{(k)} - V^{(k)} ]_{\tau}  \cdot [{\Phi_{j}} - ({\Phi_{j}})_{\mathcal S}]_{\tau} \dS \\
=  & \,  \mathrm{Re} \int_{\Omega}(u^{(k)} - V^{(k)}) \cdot \nabla {\Phi_{j}} \cdot u^{(k)} \dx  -  \mathrm{Re}  \int_{\Omega} ( \omega^{(k)} \times u^{(k)} ) \cdot {\Phi_{j}} \dx
  + \mathrm{Re} \, m \, \xi^{(k)} \times \omega^{(k)} \cdot a_{{\Phi_{j}}}  \\
 &  + \mathrm{Re} \, J_0 \omega^{(k)} \times \omega^{(k)} \cdot b_{{\Phi_{j}}}
  + \mathrm{Re} \int_{\Omega} \widetilde{[v_*]_n} \cdot \nabla {\Phi_{j}} \cdot u^{(k)} \dx  + \mathrm{Re} \int_{\Omega}(u^{(k)} - V^{(k)} ) \cdot \nabla {\Phi_{j}} \cdot \widetilde{[v_*]_n} \dx \\
 &  - \mathrm{Re} \int_{\Omega} ( \omega_k \times \widetilde{[v_*]_n}) \cdot {\Phi_{j}} \dx - \mathrm{Re} \int_{\partial \Omega} (v_{*}\cdot n)[u]_{\tau}  \cdot [{\Phi_{j}} - ({\Phi_{j}})_{\mathcal S}]_{\tau} \dS \\
  & + \mathrm{Re} \int_{\Omega} \widetilde{[v_*]_n} \cdot \nabla {\Phi_{j}} \cdot \widetilde{[v_*]_n} \dx   - 2 \int_{\Omega} {\mathbb D}(\widetilde{[v_*]_n}): {\mathbb D}({\Phi_{j}}) \dx   
  +  \alpha \int_{\partial \Omega} [v_*]_{\tau}  \cdot [{\Phi_{j}} - ({\Phi_{j}})_{\mathcal S}]_{\tau}  \dS.
\end{aligned}
\end{equation}
Let $F:{\mathbb R}^k \rightarrow {\mathbb R}^k$ be defined by
\begin{align*} 
& F_j(c^{(k)}) = F_j(c_{1}^{(k)},...,c_{k}^{(k)}) \\
= & \,  2 \int_{\Omega} {\mathbb D}({u^{(k)}}): {\mathbb D}({\Phi_{j}})\ dx + \alpha \int_{\partial \Omega} [u^{(k)} -V^{(k)} ]_{\tau}  \cdot [{\Phi_{j}} - ({\Phi_{j}})_{\mathcal S}]_{\tau}\ dS  \\
&   - \mathrm{Re} \int_{\Omega}(u^{(k)} - V^{(k)}) \cdot \nabla {\Phi_{j}} \cdot u^{(k)} dx  +  \mathrm{Re}  \int_{\Omega} ( \omega^{(k)} \times u^{(k)} ) \cdot{\Phi_{j}} dx
  - \mathrm{Re} \, m \, \xi^{(k)} \times \omega^{(k)} \cdot a_{{\Phi_j}} \\ 
  & - \mathrm{Re} \, J_0 \omega^{(k)} \times \omega^{(k)} \cdot b_{{\Phi_j}} 
  -  \mathrm{Re} \int_{\Omega} \widetilde{[v_*]_n} \cdot \nabla {\Phi_j} \cdot u^{(k)} dx - \mathrm{Re} \int_{\Omega}(u^{(k)} - V^{(k)}) \cdot \nabla {\Phi_j} \cdot \widetilde{[v_*]_n} dx \\
 &  + \mathrm{Re} \int_{\Omega} ( \omega_k \times \widetilde{[v_*]_n}) \cdot {\Phi_j} dx + \mathrm{Re} \int_{\partial \Omega} (v_{*}\cdot n)[u]_{\tau}  \cdot [{\Phi_j} - ({\Phi_j})_{\mathcal S}]_{\tau}\dS  - \mathrm{Re} \int_{\Omega} \widetilde{[v_*]_n} \cdot \nabla {\Phi_j} \cdot \widetilde{[v_*]_n} dx\\ 
 & + 2 \int_{\Omega} {\mathbb D}(\widetilde{[v_*]_n}): {\mathbb D}({\Phi_j})\ dx 
  -  \alpha \int_{\partial \Omega} [v_*]_{\tau}  \cdot [{\Phi_j} - ({\Phi_j})_{\mathcal S}]_{\tau}\  dS, \quad j =1,...,k.
\end{align*}
so that 
\[
F(c^{(k)}) = 0 \iff c^{(k)} \text{ solves system } \eqref{sistGR}.
\]

Taking into account the identities  
\[
\omega^{(k)} \times u^{(k)} \cdot u^{(k)} = 0, \qquad \xi^{(k)} \times \omega^{(k)} \cdot \xi^{(k)}= 0 = (J_0 \omega^{(k)}) \times \omega^{(k)} \cdot \omega^{(k)} ,
\]
and integrating by parts, we get
\[
\begin{aligned}
 F(c^{(k)}) \cdot c^{(k)}
=  & \, 2 \| {\mathbb D}(u^{(k)})\|_{2,\Omega_R}^2 + \alpha \|  [u^{(k)} - V^{(k)}]_{\tau} \|^2_{2,\partial \Omega} 
 +  2  {\mathrm{Re}} \int_{\Omega_R}  \widetilde{[v_*]_n} \cdot {\mathbb W}(u^{(k)}) \cdot  (u^{(k)} - V^{(k)})  \dx   \\
  &  +  2 \int_{\Omega_R} {\mathbb D}(\widetilde{[v_*]_n}): {\mathbb D}(u^{(k)}) \dx  - {\mathrm{Re}}  \int_{\Omega_R} {\widetilde{[v_*]_n}} \cdot \nabla u^{(k)} \cdot {\widetilde{[v_*]_n}} \dx  - \alpha \int_{\partial \Omega} [v_*]_{\tau}  \cdot [u^{(k)} - V^{(k)}]_{\tau} \dS  \\
 = & \, 
 \| u^{(k)} \|^2_{{\mathcal V}_R}   +   2  {\mathrm{Re}} \int_{\Omega_R} \widetilde{[v_*]_n} \cdot {\mathbb W}(u^{(k)}) \cdot (u^{(k)} - V^{(k)})  \dx 
   +  2 \int_{\Omega_R} {\mathbb D}(\widetilde{[v_*]_n}): {\mathbb D}(u^{(k)})\dx  \\
   & - {\mathrm{Re}}  \int_{\Omega_R} {\widetilde{[v_*]_n}} \cdot \nabla u^{(k)} \cdot {\widetilde{[v_*]_n}} \dx  - \alpha \int_{\partial \Omega} [v_*]_{\tau}  \cdot [u^{(k)} - V^{(k)}]_{\tau}\dS .
\end{aligned}
\]
Observe that
\begin{multline}\label{est:F1}
 \left| 2 \int_{\Omega_R} {\mathbb D}(\widetilde{[v_*]_n}): {\mathbb D}(u^{(k)})\dx  -  {\mathrm{Re}}  \int_{\Omega_R} {\widetilde{[v_*]_n}} \cdot \nabla u^{(k)} \cdot {\widetilde{[v_*]_n}} \dx  - \alpha \int_{\partial \Omega} [v_*]_{\tau}  \cdot [u^{(k)} - V_k]_{\tau} \dS \right| \\
\leq C(\Omega) \| \widetilde{[v_*]_n} \|_{1,2,\Omega}  \| u^{(k)} \|_{{\mathcal V}_R} + C(\Omega) {\mathrm{Re}} \| \widetilde{[v_*]_n} \|_{1,2,\Omega}^2  \| u^{(k)}  \|_{{\mathcal V}_R} + \alpha \| [v_*]_{\tau} \|_{2,\partial \Omega} \| u^{(k)} \|_{{\mathcal V}_R} .
\end{multline}
In view of Lemma \ref{lemaext} and estimate \eqref{est:extension}, we obtain
\begin{equation}
\label{est:F2}
\begin{aligned}
 & 2  {\mathrm{Re}} \int_{\Omega_R} \widetilde{[v_*]_n} \cdot {\mathbb W}(u^{(k)}) \cdot  (u^{(k)}  - V^{(k)} )  \dx  \\
\leq & \, {\mathrm{Re}}  [C({\mathcal B}_0) \| (v_{*}\cdot n )n - \Phi \sigma|_{\partial\Omega} \|_{1/2,2,\partial\Omega} + C({\mathcal B}_0) | \Phi  |  ]  \| u^{(k)} \|_{{\mathcal V}_R}^2.
\end{aligned}
\end{equation}
As \[
 \| u^{(k)} \|^2_{{\mathcal V}_R} = \sum_{i=1}^k \left({c_{i}^{(k)}}\right)^2 = |c^{(k)}|^2,  
\]
using \eqref{est:F1}--\eqref{est:F2}, we arrive at the following inequality:
\begin{multline*}
F(c^{(k)}) \cdot c^{(k)}
\geq \left( 1- {\mathrm{Re}}  \left[C({\mathcal B}_0) \| (v_{*}\cdot n )n - \Phi \sigma|_{\partial\Omega} \|_{1/2,2,\partial\Omega} + 
C({\mathcal B}_0) | \Phi  |  \right]   \right) |c^{(k)}|^2  \\
 - \left[ C(\Omega) \| \widetilde{[v_*]_n} \|_{1,2,\Omega}  + {\mathrm{Re}} \, C(\Omega) \| \widetilde{[v_*]_n} \|_{1,2,\Omega}^2 ) + \alpha \| [v_*]_{\tau} \|_{2,\partial \Omega} \right] |c^{(k)}|. 
\end{multline*}
Now, if 
\begin{equation}\label{smalldata}
{\mathrm{Re}}  \left[C({\mathcal B}_0) \| (v_{*}\cdot n )n - \Phi \sigma|_{\partial\Omega} \|_{1/2,2,\partial\Omega} + 
C({\mathcal B}_0) | \Phi  |  \right] < 1,
\end{equation}
which holds true for suitable $C_1 >0$ and $C_2 >0$ such that
\begin{equation}\label{smalldata2}
\mathrm{Re}|\Phi|< C_1, \quad \mathrm{Re}\| (v_{*}\cdot n )n  -  \Phi \sigma|_{\partial \Omega}\|_{H^{1/2}(\partial\Omega)} < C_2
\end{equation}
then
\[
F(c^{(k)}) \cdot c^{(k)} > 0  
\]
for all $c^{(k)} \in {\mathbb R}^k$ with
\[
|c^{(k)}| = 2 \frac{C(\Omega) \| \widetilde{[v_*]_n} \|_{1,2,\Omega}  + {\mathrm{Re}} \, C(\Omega) \| \widetilde{[v_*]_n} \|_{1,2,\Omega}^2  + \alpha \| [v_*]_{\tau} \|_{2,\partial \Omega} }{ 1- {\mathrm{Re}}  \left[C({\mathcal B}_0) \| (v_{*}\cdot n )n - \Phi \sigma|_{\partial\Omega} \|_{1/2,2,\partial\Omega} + 
C({\mathcal B}_0) | \Phi  |  \right]  }.
\]
Hence, there exists $\overline c^{(k)} \in {\mathbb R}^k$ such that $F(\overline c^{(k)}) = 0$, that is, system \eqref{sistGR} has a solution. Since $F(\overline c^{(k)} ) \cdot \overline c^{(k)} = 0$, using \eqref{est:F1}--\eqref{est:F2},
 the solution to system \eqref{sistGR} satisfies
\begin{multline*}
 \| u^{(k)} \|^2_{{\mathcal V}_R}
\leq {\mathrm{Re}}  \left[C({\mathcal B}_0) \| (v_{*}\cdot n )n - \Phi \sigma|_{\partial\Omega} \|_{1/2,2,\partial\Omega} + 
C({\mathcal B}_0) | \Phi  |  \right]   \| u^{(k)} \|^2_{{\mathcal V}_R}  \\
+ \left[ C(\Omega) \| \widetilde{[v_*]_n} \|_{1,2,\Omega}  + {\mathrm{Re}} \, C(\Omega) \| \widetilde{[v_*]_n} \|_{1,2,\Omega}^2 ) + \alpha \| [v_*]_{\tau} \|_{2,\partial \Omega} \right]  \| u^{(k)} \|_{{\mathcal V}_R}. 
 \end{multline*}
Under the assumption \eqref{smalldata}, the following uniform estimates hold:
\begin{equation}
 \| u_R^{(k)} \|_{{\mathcal V}_R} \leq \frac{C(\Omega) \| \widetilde{[v_*]_n} \|_{1,2,\Omega}  + {\mathrm{Re}} \, C(\Omega) \| \widetilde{[v_*]_n} \|_{1,2,\Omega}^2 + \alpha \| [v_*]_{\tau} \|_{2,\partial \Omega} }{ 1- {\mathrm{Re}}  \left[C({\mathcal B}_0) \| (v_{*}\cdot n )n - \Phi \sigma|_{\partial\Omega} \|_{1/2,2,\partial\Omega} + 
C({\mathcal B}_0) | \Phi  |  \right]  }, \quad \forall\ k \in {\mathbb N}, \, \forall\ R >  \delta({\mathcal B}_0).  
\label{est:ck}
\end{equation}

As the right-hand side of estimate \eqref{est:ck} is independent of $k$ and $R$, we can pass to the limits $k\rightarrow\infty$ and $R\rightarrow\infty$. {After letting \(k \to \infty\) with \(R\) fixed, we obtain a function
\(u_R \in \mathcal V_R\).
We then extend \(u_R\) by zero outside the domain \(\Omega_R\).
Passing to the limit as \(R \to \infty\) yields a function
\(u \in \mathcal V\), which satisfies an estimate analogous to
\eqref{est:ck}. Consequently, we establish the existence of a weak solution under the
restrictions that \(|\Phi|\) is sufficiently small and that $
\bigl\| (v_{*}\cdot n)n - \Phi \sigma \bigr\|_{1/2,2,\partial\Omega}$
is sufficiently small.
If \eqref{smalldata2} holds, then by \eqref{esttvn} the function $
v = u + \widetilde{[v_*]_n}$
satisfies the estimate stated in the theorem.
}
\end{proof}


\section{Nonzero Propelling Boundary Velocity}\label{sec4}

\subsection{Auxiliary Stokes problems}
 In order to solve the linear problem \eqref{eqn1tan}, we will need a number of auxiliary classical exterior Stokes problems with Navier boundary conditions. We recall the elementary rigid motion as in \eqref{elrigid} and associated  auxiliary Stokes solutions $(H^{(i)},P^{(i)})$ as in \eqref{astok}. For lifting the boundary values $v_*$, we consider
\begin{equation}
\label{liftv*}
\left\{
    \begin{aligned}
 \nabla \cdot  {\mathbb T}(\vartheta,\pi)=0,\quad 
  \nabla \cdot  \vartheta = 0 
\quad \mbox{ in }\Omega,  \\  \vartheta \cdot n
= v_* \cdot n \quad \mbox{ on }\partial\Omega, 
 \\  2 [{\mathbb D}(\vartheta) n] \times n + \alpha \vartheta \times n 
= \alpha v_{*} \times n \quad \mbox{ on }\partial\Omega,  \\ \displaystyle \lim_{\left| x\right| \rightarrow \infty }\vartheta(x)=0.  
\end{aligned}
\right.
\end{equation}
Let us establish the wellposedness results for problems \eqref{liftv*} and \eqref{astok}.

\begin{lemma}\label{lehi}
Let $\partial\Omega $ be of class $C^{2,1}$ and
let $v_{*}\cdot n \in H^{3/2}(\partial\Omega)$, $v_{*}\times n \in H^{1/2}(\partial\Omega)^3$. Then problem \eqref{liftv*}
 admits one and only one solution $(\vartheta,\pi)$ such that  
 $$
\begin{array}{l}
\vartheta \in 
H^{2,1}_{\rho}(\Omega)^3, \quad 
\pi \in H^{1,1}_{\rho} (\Omega),  
\end{array}
$$
and the following estimates hold
\begin{equation}\label{st-theta}
\|(1+|x|^2)^{-\frac{1}{2}}\vartheta\|_{L^{2}(\Omega)}+\| \nabla\vartheta \|_{H^{1}(\Omega)} + \| \pi \|_{H^{1}(\Omega)} \leq C (\|v_{*}\cdot n\|_{H^{3/2}(\partial\Omega)}+ \|v_{*}\times n\|_{H^{1/2}(\partial\Omega)})
\end{equation}
with $C=C({\mathcal{B}}).$ Moreover, for each $i \in \{1,...,6\},$ problem \eqref{astok} has a unique solution $(H^{(i)},P^{(i)})$ such that $H^{(i)},P^{(i)} \in C^{\infty}(\Omega),$ 
\begin{eqnarray*}
H^{(i)}\in H^{2,1}_{\rho}(\Omega)^3, \quad 
P^{(i)} \in H^{1,1}_{\rho} (\Omega),
\end{eqnarray*}
and it obeys the following
estimates 
\begin{equation}\label{est:H}
\|(1+|x|^2)^{-\frac{1}{2}}H^{(i)}\|_{L^{2}(\Omega)}+\| \nabla H^{(i)}\|_{H^{1}(\Omega)} + \| P^{(i)} \|_{H^{1}(\Omega)}\leq C.
\end{equation}
\end{lemma}
\begin{proof} We know from the existence and regularity result \cite[Theorem 3.6]{DhMeRa} that if we consider $v_{*}\cdot n \in H^{3/2}(\partial\Omega)$, $v_{*}\times n \in H^{1/2}(\partial\Omega)^3$, then the solution of the problem \eqref{liftv*} satisfies $(\vartheta,\pi) \in 
H^{2,1}_{\rho}(\Omega)^3, \times H^{1,1}_{\rho} (\Omega)$. The uniqueness of the solution follows from the characterization of the kernel described in \cite[Proposition 3.5]{DhMeRa} and by observing that the kernel is trivial, we obtain the desired estimate \eqref{st-theta} from \cite[Theorem 3.6]{DhMeRa}. Similarly, we can prove the existence of a unique solution $(H^{(i)},P^{(i)})$ to problem \eqref{astok} along with the estimate \eqref{est:H}.
\end{proof}

\subsection{Auxiliary fields related to rigid body}\label{Aux:rigid}

We will need the following lemma (see \cite{GaRev}).
\begin{lemma}\label{legraab}
Let $\Omega$ be locally Lipschitz and $ u \in H^{1,0}_\rho(\Omega)^3$ such that $\nabla \cdot u = 0$ in $\Omega$ and $u(x) = a_u  + b_u \times x $ on $\partial\Omega$. Then there exists $C_i=C_i(\Omega)$, $i=1,2$, such that
\begin{equation*}
\|\nabla u \|_{2,\Omega} \leq C_1 \|{\mathbb D}(u)\|_{2,\Omega}, \qquad 
|a_u| + |b_u| \leq C_2 \|{\mathbb D}(u)\|_{2,\Omega}.
\end{equation*}
\end{lemma}
In what follows, we will also use a  cut-off function defined in the following way. Let $\psi \in C_0^\infty({\mathbb R})$ be such that $0 \leq  \psi (t) \leq 1$, for all $t \in {\mathbb R}$, and 
\begin{equation*}
\psi (t) = \begin{cases}  1, \quad |t| \leq 1 \\  0, \quad |t| \geq 2 . \end{cases}
\end{equation*}
 For $R > \delta({\mathcal B}_0)$, let
\begin{equation}
\psi_R(x):=\psi(|x|/R).
\label{coff}
\end{equation} 
Then we have
\begin{equation}
\|\nabla\psi_R\|_{q,\mathbb R^3} \leq CR^{-1+3/q} \qquad (3\leq q\leq\infty).
\label{cut-est}
\end{equation}

Recall $M \in {\mathbb R}^{6\times 6}$ from \eqref{mat}:
\begin{equation*}
M_{ij}= \int_{\partial\Omega} \tilde{e}_i \cdot  {\mathbb T}(H^{(j)},P^{(j)})  n \dS, \quad i,j=1,...,6.
\end{equation*}
Also recall $K,R,S \in {\mathbb R}^{3\times 3}$ from \eqref{def:KRS} and we can write $M$ in the following form
$$
M = \bigg[ \begin{array}{cc}K & S^\top \\ S & R  \end{array}\bigg].
$$
Let us define
$$
g^{(i)}:={\mathbb T}(H^{(i)},P^{(i)})  n|_{\partial\Omega }, \quad i=1,\cdots,6.
$$
\begin{lemma}\label{M:invertible}
The matrices $M,K$ and $R$ are symmetric and positive definite. Moreover, 
$$
 M^{-1}=\bigg[ \begin{array}{cc}A&-A S^{\top} R^{-1}\\ -B SK^{-1}& B  \end{array}\bigg]
$$
where
$A:=(K-S^{\top}R^{-1}S)^{-1}$ and $B:=(R-SK^{-1}S^{\top})^{-1}.$
\end{lemma}
\begin{proof} We use the boundary conditions \eqref{astok}$_{3-4}$ to obtain the following expression of $H^{(i)}$ on the boundary $\partial\Omega$:
\begin{equation}\label{Hibdry}
\begin{aligned}
H^{(i)}  =  (\widetilde{\textsf{e}_i}  \cdot n) n + n \times (\widetilde{\textsf{e}_i}  \times n) - \frac{2}{\alpha} [n \times ([{\mathbb D}(H^{(i)}) n] \times n) ] 
 =  \widetilde{\textsf{e}_i}  - \frac{2}{\alpha} [{\mathbb D}(H^{(i)}) n]_{\tau}.
\end{aligned}
\end{equation}

Using the cut-off function \eqref{coff} to perform integration by parts over $\Omega$, we obtain
\begin{multline*}
0  =   
\displaystyle \int_{\Omega} \psi_R [ \nabla \cdot  {\mathbb T}(H^{(j)},P^{(j)}) ] \cdot  H^{(i)} \dx  =  
\displaystyle \int_{\partial \Omega} \psi_R [ {\mathbb T}(H^{(j)},P^{(j)}) n ] \cdot H^{(i)} \dS \\
- 2  \int_{\Omega}  \psi_R  {\mathbb D}(H^{(j)}) : {\mathbb D}(H^{(i)}) \dx  \displaystyle - \int_{\Omega}   [ 2  {\mathbb D}(H^{(j)}) - P^{(j)} \mathbb{I} ] : [H^{(i)} \otimes \nabla \psi_R] \dx,
\end{multline*}
and as $\psi_R = 1$ in a neighborhood of $\partial \Omega$, we get
\begin{equation}
\label{giR}
\int_{\partial \Omega} g^{(j)} \cdot H^{(i)} \dS
 = \, 2  \int_{\Omega}  \psi_R  {\mathbb D}(H^{(j)}) : {\mathbb D}(H^{(i)}) \dx  + \int_{\Omega}   (2 {\mathbb D}(H^{(j)}) - P^{(j)} \mathbb{I}): ( H^{(i)} \otimes \nabla \psi_R) \dx .
\end{equation}
Since $H^{(i)}\in H^{2,1}_{\rho}(\Omega)^3$, in particular $H^{(i)} \in L^6(\Omega)$, and $P^{(i)} \in H^{1,1}_{\rho} (\Omega)$, the last integral satisfies
\[
\begin{aligned}
& \left| \int_{\Omega}   (2  {\mathbb D}(H^{(j)}) - P^{(j)} \mathbb{I}): (H^{(i)} \otimes \nabla \psi_R) \dx \right| \\
\leq & \,  \| \nabla \psi_R \|_{3,B_{R,2R}}\left(2 \| {\mathbb D}(H^{(j)})\|_{2,\Omega}+\| P^{(j)} \|_{2,\Omega}\right) \| H^{(i)} \|_{6,\Omega} \to 0, \quad \text{as} \ R\to \infty,
\end{aligned}
\]
where $B_{R,2R}:= \{x \in {\mathbb R}^3 \mid R < |x| < 2R \}$. Letting $R\to \infty$ in \eqref{giR} yields 
\begin{equation}
\int_{\partial \Omega} g^{(j)} \cdot H^{(i)} \dS
= 2 \int_{\Omega}  {\mathbb D}(H^{(j)}) :  {\mathbb D}(H^{(i)}) \dx.
\label{rgdhi}
\end{equation}
From \eqref{Hibdry} and \eqref{rgdhi}, we get 
$$
\begin{array}{rcl}
M_{ij} & = & \displaystyle \int_{\partial\Omega} H^{(i)} \cdot  {\mathbb T}(H^{(j)},P^{(j)}) n \dS + \frac{2}{\alpha}  \int_{\partial\Omega}  [{\mathbb D}(H^{(i)}) n]_{\tau} \cdot  {\mathbb T}(H^{(j)},P^{(j)}) n \dS \medskip \\ & = & \displaystyle \int_{\partial \Omega} g^{(j)} \cdot H^{(i)} \dS + \frac{4 }{\alpha}  \int_{\partial\Omega} [{\mathbb D}(H^{(i)}) n]_\tau \cdot  [{\mathbb D}(H^{(j)}) n]_\tau \dS \medskip \\
& = & \displaystyle 2 \int_{\Omega}  {\mathbb D}(H^{(j)}) :  {\mathbb D}(H^{(i)}) \dx + \frac{4 }{\alpha}  \int_{\partial\Omega} [{\mathbb D}(H^{(i)}) n]_\tau \cdot  [{\mathbb D}(H^{(j)}) n]_\tau \dS.
\end{array}
$$
Recalling \eqref{Hibdry} and the inner product defined in the space $\mathcal{V}$, we can write
\begin{equation}
\label{mijs}
M_{ij} = 2 \int_{\Omega} {\mathbb D}(H^{(i)}):{\mathbb D}(H^{(j)}) \dx + \alpha  \int_{\partial\Omega} [ H^{(i)} - \widetilde{\textsf{e}_i} ]_\tau \cdot 
[ H^{(j)} - \widetilde{\textsf{e}_j} ]_\tau  \dS = (H^{(i)},H^{(j)})_{\mathcal{V}}.
\end{equation}

It is clear that the matrices $M$, $K$ and $R$ are symmetric. In order to  prove that $M$ is positive definite, we first notice that
\begin{eqnarray*}
 z_i M_{ij} z_j & = &  (z_i H^{(i)}, z_j H^{(j)})_{\mathcal V} \\
& = & 2 \| D(z_i H^{(i)}) \|_{2,\Omega}^2 + \alpha  \| [z_i H^{(i)} - z_i \widetilde{\textsf{e}_i}]_\tau \|_{2,\partial\Omega}^2 \geq 0
\end{eqnarray*}
for all $z=(z_1,z_2,...,z_6) \in {\mathbb R}^6.$ Suppose that 
\begin{equation}\label{Dw0}
z_i M_{ij}z_j = 0 \iff  \| {\mathbb D}(z_i H^{(i)}) \|_{2,\Omega} = \| [z_i H^{(i)} - z_i \widetilde{\textsf{e}_i}]_\tau \|_{2,\partial\Omega} = 0.
\end{equation}
Let us set $w:=z_iH^{(i)}$ in $\Omega$ and $W:=z_i\widetilde{\textsf{e}_i}$. As $ [z_i H^{(i)} - z_i \widetilde{\textsf{e}_i}]_\tau  = [ W ]_\tau  = 0$ on $\partial\Omega$, we use the identity \eqref{Hibdry} to conclude that $w = W$ on $\partial\Omega$. Also $w\in H^{1,0}_{\rho}$, so we can apply Lemma \ref{legraab} to obtain
$$
| z_i | \leq C \| {\mathbb D}(w)\|_{2,\Omega} = 0, \quad i=1,...,6,
$$
Consequently, $M$ is positive definite. Thus, $K$ and $R$ are positive definite. 
\end{proof}

Exploiting the regularity properties of Lemma \ref{lehi} and the first self-propelling condition, we prove the existence result for the problem \eqref{eqn1tan}.
 \begin{proof}[Proof of Theorem \ref{Mainresult1}]
The solution of problem \eqref{eqn1tan} is obtained by solving a linear system. Set $(v,p)=(c_i H^{(i)} + \vartheta, c_i P_i + \pi ),$ with $c =(c_1,c_2,...,c_6) \in {\mathbb R}^6$, $(H^{(i)},P^{(i)})$ and  $(\vartheta,\pi)$ defined in \eqref{astok} and \eqref{liftv*}, respectively.  Then $(v,p)$ satisfy $(\ref{eqn1tan})_{1-4},$ for any $c \in {\mathbb R}^6.$ In order to satisfy the self-propelling conditions $(\ref{eqn1tan})_{5},$ we impose that $c$ solves the following system
\begin{equation}
M c = \beta
\label{ls}
\end{equation}
where 
\begin{equation*}
\beta_i=-\int_{\partial\Omega} \widetilde{\textsf{e}_i}  \cdot  {\mathbb T}(\vartheta,\pi)  n \dS,\quad i=1,...,6. 
\label{betai}
\end{equation*}
{We can solve $c \in\mathbb{R}^6$ uniquely in the equation \eqref{ls} due to the invertibility of the matrix $M$ as shown in Lemma \ref{M:invertible}. 
Consequently, $(v,p)=(c_i H^{(i)} + \vartheta,c_i P_i + \pi )$ with $c$ given by $c=M^{-1}\beta$ solves \eqref{eqn1tan}.} Using (\ref{ls}), we have
{$$
\|M\|^{-1} |\beta| \leq | c | \leq \|M^{-1}\| |\beta|
$$}
and since $\beta = {\mathbb P} (v_{*}),$ it holds
\begin{equation*}
C_1 \| {\mathbb P} (v_{*}) \| \leq |\xi  | + |\omega | \leq C_2 \|{\mathbb P} (v_{*})\|
\label{eqco} 
\end{equation*}
with $C_i=C_i(\mathcal{B}_0), i=1,2.$ Concerning uniqueness, suppose that $(v^{(i)},V^{(i)},p^{(i)}) \in H^{2,1}_{\rho}(\Omega)^3 \times H^{1,1}_{\rho} (\Omega), i=1,2,$ are two solutions. Then $(v,V,p)=(v^{(1)} - v^{(2)},V^{(1)}-V^{(2)},p^{(1)}-p^{(2)})$ solves 
\begin{equation}
\label{eqnuni}\left\{ 
\begin{aligned}
 \nabla \cdot  {\mathbb T}(v,p)  =0, \quad  
 \nabla \cdot  v  = 0  \quad \text{ in }\Omega, \\
 \displaystyle v \cdot n
 = V \cdot n \quad \text{ on }\partial\Omega, 
 \\  2 [{\mathbb D}(v) n] \times n + \alpha v \times n 
 = \alpha V \times n \quad \text{ on }\partial\Omega,  \\ 
 \lim_{\left| x\right| \rightarrow
\infty }v(x)  = 0, \\  \int_{\partial\Omega} {\mathbb T}(v
,p) n dS  =  \int_{\partial\Omega} x \times {\mathbb T}(v,p)  n dS =0 .
\end{aligned}
\right. 
\end{equation}

Recall the cut-off function \eqref{coff}. Dot-multiplying $\eqref{eqnuni}_1$ by $\psi_R v$ and integrating by parts over $\Omega,$ we get
\begin{equation*}
 \int_{\partial \Omega} [ {\mathbb T}(v,p) n ] \cdot v  \dS
- 2  \int_{\Omega}  \psi_R | {\mathbb D}(v) |^2 \dx
 - \int_{\Omega}   [ 2  {\mathbb D}(v) - p \mathbb{I} ] : [ v \otimes \nabla \psi_R] \dx=0 .
\end{equation*}
Letting $R\to \infty$ in the above relation yields   
\begin{equation}
  \int_{\partial \Omega} [ {\mathbb T}(v,p) n ] \cdot v  \dS = 2 \|{\mathbb D}(v)\|_{2,\Omega}^2.
  \label{uniqst}
\end{equation}
From the boundary conditions, it follows
\begin{equation}
v =  V - \frac{2}{\alpha} [{\mathbb D}(v) n]_\tau \mbox{ on }\partial\Omega. 
\label{vVD}
\end{equation}
Using \eqref{vVD} in \eqref{uniqst}, we get
\[
\int_{\partial\Omega} V \cdot {\mathbb T}(v,p)  n \dS - \frac{2}{\alpha} \int_{\partial\Omega} [{\mathbb D}(v) n]_\tau \cdot {\mathbb T}(v,p) n \dS = 2 \|{\mathbb D}(v)\|_{2,\Omega}^2.
\] 
Since 
\[
[{\mathbb D}(v) n]_\tau \cdot {\mathbb T}(v,p) n  = | [{\mathbb D}(v) n]_\tau |^2 = \frac{\alpha^2}{4} | (v -  V)|_{\partial \Omega} |^2
\]
we obtain (recall that  $V(x)=\xi + \omega \times x$)
$$
\xi \cdot \int_{\partial\Omega} {\mathbb T}(v,p)  n \dS + \omega \cdot \int_{\partial\Omega}  x \times {\mathbb T}(v,p)  n \dS = \frac{\alpha}{2} | (v -  V)|_{\partial \Omega} |^2 +  2 \|{\mathbb D}(v)\|_{2,\Omega}^2.
$$
Due to the self-propelling conditions $\eqref{eqnuni}_{5}$,  we get
$$
v  =  V \text{ on } \partial \Omega , \quad  {\mathbb D}(v) = 0 \text{ in } \Omega.
$$
From Lemma \ref{legraab}, we deduce
\[
|\xi| + |\omega| \leq C \|{\mathbb D}(v)\|_{2,\Omega} = 0,
\]
so that $V=0$ and $(v,p)$ satisfies
\[\left\{
    \begin{aligned}
\displaystyle 
\displaystyle \nabla \cdot  {\mathbb T}(v,p) = 0, \quad 
\nabla \cdot  v  = 0  \mbox{ in } \Omega,  \\
\displaystyle v = 0   \mbox{ on } \partial\Omega,  \\ 
\displaystyle \lim_{\left| x\right| \rightarrow \infty }v(x)= 0.
 \end{aligned}
\right.
\]
From the summability properties of $v$ and $p$ and classical results for the exterior Stokes problem with Dirichlet boundary conditions, it follows that $v=0$ and $p=0.$
\end{proof}

\subsection{The relation between the propelling velocity and the velocity of rigid body}

Based on the results of \cite{Galdi1999} and the regularity results from the Lemma \ref{lehi}, one may expect that when $\partial\Omega$ is of class $C^{2,1}$, the set $\{g^{(1)},...,g^{(6)}\} \subset L^2(\partial\Omega)$, with $g^{(i)}:={\mathbb T}(H^{(i)},P^{(i)})  n |_{\partial\Omega }$, is linearly independent. We introduce the \emph{thrust space} 
\begin{equation}\label{thrust}
{\mathcal{T}}({\mathcal{B}_0}):=\mbox{span}\{g^{(1)},...,g^{(6)}\}
\end{equation}
and the projection operator ${\mathbb P}$ in ${\mathcal{T}}({\mathcal{B}_0}
).$ We have the following result
\begin{lemma}
Let $\partial\Omega$ be of class $C^{2,1}$. Then, the system of vector functions 
$\left\{ g^{(1)},... ,g^{(6)} \right\} \subset L^2(\partial \Omega)$
is linearly independent.
\label{lemali}
\end{lemma}
\begin{proof} Let $c_k \in \mathbb R$, $k=1,...,6$, be such that $c_k g^{(k)} = 0$ in $L^2(\partial \Omega)$. Here and throughout the proof, we use the Einstein summation convention. Our aim is to show that $c_k = 0$ for all $k \in \{1,...,6\}$. Setting $H:= c_k H^{(k)}$ and $P:= c_k P^{(k)}$, we have that $(H,P) \in H^{2,1}_{\rho}(\Omega)^3 \times H^{1,1}_{\rho} (\Omega)$ satisfies
\[
\left\{
\begin{aligned}
 \nabla \cdot  {\mathbb T}(H,P)  = 0,\quad \nabla \cdot H  = 0   \text{ in }\Omega,  \\
 H \cdot n  =  c_k \widetilde{\textsf{e}_k} \cdot n \mbox{ on }\partial\Omega,  \\ 
\displaystyle 2 [{\mathbb D}(H) n] \times n + \alpha H \times n 
 =  c_k \widetilde{\textsf{e}_k} \times n \mbox{ on }\partial\Omega,  \\  
 {\mathbb T}(H,P)  n  = 0 \text{ on } \partial\Omega,  \\ 
\displaystyle \lim_{\left| x\right| \rightarrow \infty }H(x)  =0 .
\end{aligned}
\right.
\]
Let $H_{\mathcal S}: = c_k \widetilde{\textsf{e}_k} \in {\mathcal R}$. Since ${\mathbb T}(H,P)  n = 0$ on $\partial\Omega$, and consequently,
\[
H = H_{\mathcal S} - \frac{2}{\alpha} [{\mathbb D}(H) n]_{\tau} = 
H_{\mathcal S} - \frac{2}{\alpha} \left[ {\mathbb T}(H,P)  n \right]_{\tau} =  H_{\mathcal S} \text{ on } \partial\Omega,
\]
it follows that $(H,P)$ satisfies
\[
\left\{
\begin{aligned}
 \nabla \cdot  {\mathbb T}(H,P)  = 0,\quad
  \nabla \cdot H = 0   \text{ in }\Omega,  \\
 H  =  H_{\mathcal S}  \text{ on }\partial\Omega, 
 \\ 
 {\mathbb T}(H,P)  n  = 0  \text{ on } \partial\Omega, \\ 
 \lim_{\left| x\right| \rightarrow \infty }H(x)  = 0 .
\end{aligned}
\right.
\]
We use the relation
\[
\| {\mathbb D}(H) \|^2_{2,\Omega} = \int_{\partial \Omega} H_{\mathcal S} \cdot {\mathbb T}(H,P)  n \dS = 0
\]
and Lemma \ref{legraab} to obtain
\[
|c_k| \leq C \| {\mathbb D}(H) \|_{2,\Omega} = 0, \quad k=1,...,6.
\]
It implies $c_k=0$, $k=1,...,6$.
\end{proof}

\subsection{Thrust space formed by tangential motions}

In case the self-propelled motion occurs because the body moves tangentially its boundary, the \emph{thrust space} is defined as
\begin{equation}\label{thrust-tau}
{\mathcal{T}}_{\tau}({\mathcal{B}_0}):=\mbox{span}\{[g^{(1)}]_\tau,...,[g^{(6)}]_\tau\}.
\end{equation}
\begin{lemma}
Let $\partial\Omega$ be of class $C^{2,1}$. Then the system of vector functions $\left\{[g^{(1)}]_\tau,...,[g^{(6)}]_\tau\right\} \subset L^2(\partial\Omega )$
is linearly independent.
\end{lemma}
\begin{proof} Suppose  $c_k \in \mathbb R$, $k=1,...,6$, are such that $c_k [g^{(k)}]_\tau = 0$ at $\partial \Omega$. Setting $H:= c_k H^{(k)}$ and $P:= c_k P^{(k)}$, we have 
\[
\left\{
\begin{aligned} \displaystyle 
 \nabla \cdot  {\mathbb T}(H,P)  = 0,\quad
\nabla \cdot H = 0 \text{ in }\Omega,  \\
\displaystyle H \cdot n
 = c_i \widetilde{\textsf{e}_i} \cdot n  \text{ on }\partial\Omega,
 \\ 
 2 [{\mathbb D}(H) n] \times n + \alpha H \times n 
 = \alpha c_i \widetilde{\textsf{e}_i} \times n \text{ on }\partial\Omega, \\  
\left[ {\mathbb T}(H,P)  n \right]_{\tau}  = 0  \text{ on } \partial\Omega,  \\ 
\displaystyle \lim_{\left| x\right| \rightarrow
\infty }H(x)  = 0.
\end{aligned}
\right.
\]
Observe that $2 [{\mathbb D}(H) n]_{\tau} = \left[ {\mathbb T}(H,P)  n \right]_{\tau} = 0$. If we define $H_{\mathcal S}: = c_i \widetilde{\textsf{e}_i}$, we have $H = H_{\mathcal S} - \frac{2}{\alpha} [{\mathbb D}(H) n]_{\tau}$. We immediately deduce that $H =  H_{\mathcal S} \text{ on } \partial\Omega$,
and due to the fact that $\nabla H_{\mathcal S}$ is skew-symmetric, we have
\[
 {\mathbb T}(H,P)  n  = \left( n \cdot {\mathbb T}(H,P)  n \right) n =  2 \left(n \cdot \nabla H \cdot n \right) n - P n = 2 \left( n \cdot \nabla (H-H_{\mathcal S}) \cdot n \right) n - P n \text{ on } \partial\Omega.
\]
{ Let  $U:= H - H_{\mathcal S}$ and recall the relation between the surface divergence at $\partial \Omega$ and the divergence in Cartesian coordinates: $\nabla_{\tau} \cdot U   = \nabla \cdot U - n \cdot \nabla U  \cdot n \text{ on } \partial\Omega$. Since $U = 0$ on $\partial\Omega$ and $\nabla \cdot U = 0$ in $\Omega$, 
we have 
\begin{equation}\label{decompose1}
  n \cdot \nabla U \cdot n = \frac{\partial U_j}{\partial n} n_j = 0 \text{ on } \partial \Omega,
\end{equation}
and consequently $
 {\mathbb T}(H,P)  n =  - P n$ on  $\partial\Omega$}. Observe that the pressure $P$ satisfies
\begin{equation}\label{eq:P}
\left\{
\begin{aligned}
 \Delta P =0 \quad \mbox{ in }\Omega, \\
 \frac{\partial P}{\partial n} =  \Delta H \cdot n  \quad \mbox{ on }\partial\Omega,  \\ 
 \lim_{\left| x\right| \rightarrow \infty }P(x) =0. 
\end{aligned}
\right.
\end{equation}
It is obvious that $
\Delta H \cdot n  = \Delta U  \cdot n $. For each component $U_j$, the relation between the Laplace-Beltrami operator
and the Cartesian Laplacian reads  (see, for example,  \cite{Xu})
\begin{equation}\label{Deltatau}
\Delta_{\tau} U_j = \Delta U_j - (\nabla \cdot n) \frac{\partial U_j}{\partial n}   - n \cdot (\nabla \otimes \nabla U_j)  n, \quad j=1,2,3,  
\end{equation}
where the surface Laplacian on $\partial \Omega$ is denoted by $\Delta_{\tau}$ and $\nabla \otimes \nabla U_j$ stands for the Hessian matrix of $U_j$.
Since $U = 0$ on $\partial \Omega$, we have 
$\Delta_{\tau} U =  0 
$ on $\partial \Omega$. 
Note that $\left[ {\mathbb T}(U,P)  n \right]_{\tau}  = \left[ {\mathbb T}(H,P)  n \right]_{\tau} = 0$ and \eqref{decompose1} yield ${\mathbb D}(U)n=0$ on $\partial\Omega$.
The condition ${\mathbb D}(U)n=0$ is equivalent to
\begin{equation*}
\frac{\partial U_i}{\partial x_j}n_j 
  = -\,\frac{\partial U_j}{\partial x_i}n_j
  \quad\text{on } \partial\Omega, \quad i=1,2,3.
\end{equation*}
Differentiating the above relation with respect to $x_k$ gives
\begin{equation*}
\frac{\partial^2 U_i}{\partial x_k \partial x_j}n_j
  + \frac{\partial U_i}{\partial x_j}\frac{\partial n_j}{\partial x_k}
  = -\,\frac{\partial^2 U_j}{\partial x_k \partial x_i}n_j
  - \frac{\partial U_j}{\partial x_i}\frac{\partial n_j}{\partial x_k},
  \qquad i,k=1,2,3.
\end{equation*}
Multiplying this identity by $n_i n_k$ and summing over $i,k$ yields
\begin{align*}
\frac{\partial^2 U_i}{\partial x_k \partial x_j} n_i n_j n_k
  + \frac{\partial U_i}{\partial x_j}\frac{\partial n_j}{\partial x_k} n_i n_k
  &=-\,\frac{\partial^2 U_j}{\partial x_k \partial x_i} n_i n_j n_k
  - \frac{\partial U_j}{\partial x_i}\frac{\partial n_j}{\partial x_k} n_i n_k.
\end{align*}
Hence, by using ${\mathbb D}(U)n=0$ on $\partial\Omega$, we obtain
\begin{align}\label{decompose2}
2\,n \cdot (\nabla\!\otimes\nabla U_j) n\, n_j
 &= -\left(\frac{\partial U_j}{\partial x_i}
     + \frac{\partial U_i}{\partial x_j}\right)
    n_i\,\frac{\partial n_j}{\partial x_k}n_k
 = -2\,[{\mathbb D}(U)n]_j\,\frac{\partial n_j}{\partial n}
 = 0.
\end{align}
Thus, due to the decomposition \eqref{Deltatau}, the fact that $\Delta_{\tau} U =  0 
$ on $\partial \Omega$ and the relations \eqref{decompose1}, \eqref{decompose2}, we have  \begin{equation*}
\Delta U \cdot n = (\Delta_{\tau} U_j) n_j + (\nabla \cdot n) \frac{\partial U_j}{\partial n} n_j  + n \cdot (\nabla \otimes \nabla U_j) \cdot n n_j = 0 \quad \mbox{on}\quad \partial\Omega.
\end{equation*}
Thus, using the fact $
\Delta H \cdot n  = \Delta U  \cdot n $ and pressure equation \eqref{eq:P}, we conclude that {$(H,P)=(H,0)$} satisfies
\[
\left\{
\begin{aligned}
\displaystyle \nabla \cdot  {\mathbb T}(H,P)  = 0,\quad
  \nabla \cdot H = 0  \quad  \text{ in }\Omega,  \\
\displaystyle H  =  H_{\mathcal S} \quad  \text{ on }\partial\Omega,  \\ 
 {\mathbb T}(H,P)  n  = 0 \quad  \text{ on } \partial\Omega, \\ 
\displaystyle \lim_{\left| x\right| \rightarrow \infty }H(x)  = 0 .
\end{aligned}
\right.
\]
Using the relation
\[
\| {\mathbb D}(H) \|^2_{2,\Omega} = \int\limits_{\partial \Omega} H \cdot {\mathbb T}(H,P)  n \ dS = 0
\]
and Lemma \ref{legraab}, we conclude $c_k=0$, $k=1,...,6$.\end{proof}

 \begin{proof}[Proof of Theorem \ref{Mainresult2}] We multiply \eqref{eqn1tan}$_1$ by $\psi_R H^{(i)}$, where $\psi_R$ is the cut-off function \eqref{coff},  do integration by parts, and let $R \to \infty$, to obtain: 
 \begin{equation*}
 -2\int_{\Omega} \mathbb{D}(v): \mathbb{D}(H^{(i)}) \dx+  \int_{\partial \Omega}{\mathbb T}(v,p)n\cdot H^{(i)} \dS=0.
 \end{equation*}
 We use the boundary conditions \eqref{astok}$_{3-4}$ to obtain the following expression of $H^{(i)}$ as in \eqref{Hibdry}:
\begin{equation*}
H^{(i)} =  \widetilde{\textsf{e}_i}  - \frac{2}{\alpha} [{\mathbb D}(H^{(i)}) n]_{\tau} \quad\mbox{on}\quad \partial\Omega.
\end{equation*}
Thus, using the above two expressions, we get 
\begin{equation}\label{rel1}
2\int\limits_{\Omega} \mathbb{D}(v): \mathbb{D}(H^{(i)}) \dx =  \int_{\partial\Omega}  {\mathbb T}(v,p)n\cdot \widetilde{\textsf{e}_i}\dS -  \frac{2}{\alpha} \int_{\partial\Omega}  [{\mathbb D}(H^{(i)}) n]_{\tau} [{\mathbb D}(v) n]_{\tau} \dS.
\end{equation}
Now we multiply \eqref{astok}$_{1}$ by $\psi_R v$, do integration by parts, and let $R \to \infty$, to have 
 \begin{equation*}
 2\int\limits_{\Omega} \mathbb{D}(v): \mathbb{D}(H^{(i)}) \dx=  \int_{\partial\Omega}  {\mathbb T}(H^{(i)}, P^{(i)})n\cdot v \dS.
 \end{equation*}
 As in Remark \ref{BCv}, we have 
 \begin{equation*}
 v  -V =  [(v  -V)\cdot n ] n + n \times [ (v-V) \times n )] =  v_{*}  - \frac{2}{\alpha} n \times ( [{\mathbb D}(v) n] \times n).
\end{equation*}
Thus, using the above two expressions, we obtain 
\begin{equation}\label{rel2}
2\int_{\Omega} \mathbb{D}(v): \mathbb{D}(H^{(i)}) \dx = \int_{\partial\Omega} (v_{*}+V)\cdot {\mathbb T}(H^{(i)},P^{(i)})n \dS -  \frac{2}{\alpha} \int_{\partial\Omega}  [{\mathbb D}(H^{(i)}) n]_{\tau} [{\mathbb D}(v) n]_{\tau} \dS.
\end{equation}
Comparing \eqref{rel1}--\eqref{rel2}, we get
\begin{equation*}\label{rel3}
\int_{\partial\Omega}  {\mathbb T}(v,p)n\cdot \widetilde{\textsf{e}_i}\dS =  \int_{\partial\Omega}  (v_{*}+V)\cdot {\mathbb T}(H^{(i)},P^{(i)})n \dS.
\end{equation*}
Recall $g^{(i)}:={\mathbb T}(H^{(i)},P^{(i)})  n|_{\partial\Omega }$, $i=1,...,6$. Then by using the self-propelling conditions \eqref{eqn1tan}$_5$, we obtain 
\begin{equation}\label{rel4}
\int_{\partial\Omega}  (v_{*}+V)\cdot g^{(i)} \dS=0, \quad i=1,...,6.
\end{equation}
We want to rewrite the relation \eqref{rel4}. In order to do so, let us recall { $M \in {\mathbb R}^{6\times 6}$ from \eqref{mat} and \eqref{mijs}:
\begin{equation}\label{rel5}
M_{ij}= \int_{\partial\Omega} \tilde{e}_i \cdot  {\mathbb T}(H^{(j)},P^{(j)})  n \dS = (H^{(i)},H^{(j)})_{\mathcal V}, \quad i,j=1,\cdots,6.
\end{equation}}
Also recall $K,R,S \in {\mathbb R}^{3\times 3}$ from \eqref{def:KRS} and we can write $M$ in the following form
$$
M = \bigg[ \begin{array}{cc}K & S^\top \\ S & R  \end{array}\bigg].
$$
Let us define 
\begin{equation}\label{Wi}
W_i := - \int_{\partial\Omega} v_*\cdot g^{(i)} \dS,\quad i=1,\cdots ,6.
\end{equation}
As $V=\xi + \omega \times x$, using the relations \eqref{rel5} and \eqref{Wi}, we can rewrite \eqref{rel4} in the following way:
\begin{equation}
\begin{bmatrix}
(W_1,W_2,W_3)^{\top}\\
(W_4,W_5,W_6)^{\top}
\end{bmatrix}=M \begin{pmatrix}
\xi \\ \omega
\end{pmatrix}
\end{equation}
Let us denote
\begin{equation*}
A:= (K-S^{\top}R^{-1}S)^{-1},\quad B:= (R-SK^{-1}S^{\top})^{-1}.
\end{equation*}
We can solve 
\begin{align}
\xi = A \Big((W_1,W_2,W_3)- S^{\top}R^{-1}(W_4,W_5,W_6)\Big),\quad
\omega = B\Big((W_4,W_5,W_6)- SK^{-1}(W_4,W_5,W_6)\Big).\label{xi1}
\end{align}
Thus, the translational velocity $\xi$ of the rigid body is nonzero iff 
\begin{equation*}
(W_1,W_2,W_3)\neq S^{\top}R^{-1}(W_4,W_5,W_6),
\end{equation*}
and it's angular velocity $\omega$ is nonzero iff
\begin{equation*}
(W_4,W_5,W_6)\neq S K^{-1}(W_1,W_2,W_3).
\end{equation*}
\end{proof}

 \begin{proof}[Proof of Theorem \ref{selfprop-NS}]
 Under the assumptions, according to the Theorem \ref{Mainresult3}, we know that the problem \eqref{eqn1:NS} admits at least one weak solution $(v,V)\in H^{1,0}_{\rho}(\Omega) \times \mathcal{R}$. We want to show that $V\neq 0$. The idea is to compare the solution of the Stokes-rigid body system with the Navier-Stokes-rigid body system and use the result of Theorem \ref{Mainresult2}.
 
 Let $(v_0,V_0)$ be a solution of \eqref{eqn1tan} and $(v,V)$ be a solution of \eqref{eqn1:NS}. Multiply the first equation of \eqref{astok} by $\psi_R (v - v_0)$, where $v$ is the velocity field of the Navier-Stokes problem, $v_0$ is the velocity of the Stokes problem and $\psi_R$ is the cut-off function defined in \eqref{coff},  and integrate by parts:
\begin{equation}
\label{mvv0}
\begin{array}{rcl}
0 & =  & 
\displaystyle \int_{\Omega} \psi_R [ \nabla \cdot  {\mathbb T}(H^{(i)},P^{(i)}) ] \cdot  (v-v_0)  \dx
\medskip \\
  &  =  & 
\displaystyle \int_{\partial \Omega} [ {\mathbb T}(H^{(i)},P^{(i)}) n ] \cdot (v - v_0) \dS
- 2  \int_{\Omega}  \psi_R  {\mathbb D}(H^{(i)}) : {\mathbb D}(v-v_0) \dx
\medskip \\
&  & \displaystyle - \int_{\Omega}   [ 2  {\mathbb D}(H^{(i)}) - P^{(i)} \mathbb{I} ] : [(v - v_0) \otimes \nabla \psi_R] \dx .
\end{array}
\end{equation}
Recalling $ g^{(i)} : = {\mathbb T}(H^{(i)},P^{(i)}) n $, it follows from \eqref{mvv0} that
\begin{equation}
\label{giR1}
\begin{aligned}
\int_{\partial \Omega} g^{(i)} \cdot (v - v_0) \dS
 = & \, 2  \int_{\Omega}  \psi_R  {\mathbb D}(H^{(i)}) : {\mathbb D}(v-v_0) \dx \\
&  + \int_{\Omega}   (2 {\mathbb D}(H^{(i)}) - P^{(i)} \mathbb{I}): ((v - v_0) \otimes \nabla \psi_R) \dx .
\end{aligned}
\end{equation}
Recalling \eqref{cut-est} together with the summability properties given by Lemma \ref{lehi} and $v - v_0 \in L^6(\Omega)$, we get
\[
\begin{aligned}
& \left| \int_{\Omega}   (2  {\mathbb D}(H^{(i)}) - P^{(i)} \mathbb{I}): ((v - v_0) \otimes \nabla \psi_R) \dx \right| \\
\leq & \,  \| \nabla \psi_R \|_{3,B_{R,2R}}\left(2 \| {\mathbb D}(H^{(i)})\|_{2,\Omega}+\| P^{(i)} \|_{2,\Omega}\right) \|v-v_0\|_{6,B_{R,2R}} \to 0, \quad \text{as} \ R\to \infty,
\end{aligned}
\]
where $B_{R,2R}:= \{x \in {\mathbb R}^3 \mid R < |x| < 2R \}$.
Hence, letting $R\to \infty$ in \eqref{giR1}, yields 
\begin{equation}
\int_{\partial \Omega} g^{(i)} \cdot (v - v_0) \dS
= 2 \int_{\Omega}  {\mathbb D}(H^{(i)}) :  {\mathbb D}(v-v_0) \dx.
\label{rgdhi1}
\end{equation}
Now, for $V=v_{\mathcal S}$ and $V_0 = (v_0)_{\mathcal S}$, we can write
\[
\int_{\partial \Omega} g^{(i)} \cdot (v - v_0) \dS = \int_{\partial \Omega} g^{(i)} \cdot [ (v - v_0) - (V - V_0) ] \dS + \int_{\partial \Omega} g^{(i)} \cdot (V - V_0) \dS
\]
where, due to $(v - v_0) \cdot n  =  (V - V_0) \cdot n$,
\[
\int_{\partial \Omega} g^{(i)} \cdot [ (v - v_0) - (V - V_0) ] \dS = \int_{\partial \Omega} [g^{(i)}]_\tau \cdot [ (v - v_0) - (V - V_0) ]_\tau \dS.
\]
Since $[g^{(i)}]_\tau = 2 [ {\mathbb D}(H^{(i)}) n ]_\tau = \alpha [\widetilde{\textsf{e}_i}  - H^{(i)} ]_\tau$, we get
\[
\int_{\partial \Omega} g^{(i)} \cdot (v - v_0) \dS = \alpha \int_{\partial \Omega} [\widetilde{\textsf{e}_i}  - H^{(i)} ]_\tau \cdot [ (v - v_0) - (V - V_0) ]_\tau \dS + \int_{\partial \Omega} g^{(i)} \cdot (V - V_0) \dS
\]
and using \eqref{rgdhi1}, we obtain
\begin{equation}\label{VV0}
\int_{\partial \Omega} g^{(i)} \cdot (V - V_0) \dS
= 2 \int_{\Omega} {\mathbb D}(H^{(i)}) : {\mathbb D}(v-v_0) \dx + \alpha \int_{\partial \Omega} [ H^{(i)} - \widetilde{\textsf{e}_i} ]_\tau \cdot [ (v - v_0) - (V - V_0) ]_\tau \dS.
\end{equation}
Here $V=\xi + \omega \times x$ and $V_0=\xi_0 + \omega_0 \times x$. Let us set 
\begin{align*}
\widetilde{\xi}&= \xi-\xi_0,\quad \widetilde{\omega}= \omega-\omega_0,\\ 
\mathcal{F}_i = 2 \int_{\Omega} {\mathbb D}(H^{(i)}) : {\mathbb D}(v-v_0) \dx &+ \alpha \int_{\partial \Omega} [ H^{(i)} - \widetilde{\textsf{e}_i} ]_\tau \cdot [ (v - v_0) - (V - V_0) ]_\tau \dS,\quad i=1,2,\cdots 6.
\end{align*}
Keeping the above notations, we use the equation \eqref{VV0} and obtain 
\begin{align}\label{sol:wideom}
\widetilde{\xi} = A \Big((\mathcal{F}_1,\mathcal{F}_2,\mathcal{F}_3)- S^{\top}R^{-1}(\mathcal{F}_4,\mathcal{F}_5,\mathcal{F}_6)\Big),\
\widetilde{\omega} = B\Big((\mathcal{F}_4,\mathcal{F}_5,\mathcal{F}_6)- SK^{-1}(\mathcal{F}_4,\mathcal{F}_5,\mathcal{F}_6)\Big),
\end{align}
where $A:= (K-S^{\top}R^{-1}S)^{-1},\quad B:= (R-SK^{-1}S^{\top})^{-1}$. As in \cite[Theorem 6.2]{Galdi1999}, we have the following result 
\begin{equation}\label{bdF}
\forall\ \eta>0, \exists \ C>0\mbox{ such that }\mathrm{Re}< C\mbox{ implies }|(\mathcal{F}_1,\cdots,\mathcal{F}_6)|<\eta.
\end{equation}
Using the upper bounds of the entries of the matrices $A, B, S^{\top}R^{-1}, SK^{-1}$ and the result \eqref{bdF}, we obtain from the relation \eqref{sol:wideom} that
\begin{equation}\label{ineqxi}
|\xi_0|-c\eta \leq |\xi| \leq |\xi_0|+c\eta, \quad |\omega_0|-c\eta \leq |\omega| \leq |\omega_0|+c\eta.
\end{equation}
We know from \eqref{xi1} that
\begin{align}\label{xi0}
\xi_0 = A \Big((W_1,W_2,W_3)- S^{\top}R^{-1}(W_4,W_5,W_6)\Big),\quad
\omega_0 = B\Big((W_4,W_5,W_6)- SK^{-1}(W_4,W_5,W_6)\Big).
\end{align}
In particular, let us choose $c\eta=\frac{1}{2}\min\{|\xi_0|,|\omega_0|\}$. In view of \eqref{ineqxi} and \eqref{xi0}, if $(W_1,W_2,W_3)\neq S^{\top}R^{-1}(W_4,W_5,W_6)$, then 
\begin{equation*}
\frac{1}{2}|A\left((W_1,W_2,W_3) - S^{\top}R^{-1}(W_4,W_5,W_6)\right)| \leq |\xi| \leq \frac{3}{2}|A\left((W_1,W_2,W_3) - S^{\top}R^{-1}(W_4,W_5,W_6)\right)|,
\end{equation*}
while, if $(W_4,W_5,W_6)\neq S K^{-1}(W_1,W_2,W_3)$, then 
\begin{equation*}
\frac{1}{2}|B\left((W_4,W_5,W_6) - S K^{-1}(W_1,W_2,W_3)\right)| \leq |\omega| \leq \frac{3}{2}|B\left((W_4,W_5,W_6) - S K^{-1}(W_1,W_2,W_3)\right)|.
\end{equation*}
 \end{proof}

\end{document}